\documentclass[a4paper,12pt,reqno,draft]{amsart}
\usepackage{amssymb,amsmath,array,amscd,amsthm,hhline,xcolor}

\usepackage[mathscr]{euscript}
\usepackage{stmaryrd}
\usepackage{ulem}

\usepackage{mathrsfs}

\usepackage[backref,backref=false,final=true]{hyperref}

\hypersetup{bookmarks=true}

\def\SSigma_#1{\BS(#1)}

\usepackage{tikz-cd}

\usepackage{mathrsfs}

\def\({\left(}
\def\){\right)}

\def\l{\ell}
\def\Eu{\mathop{\rm Eu}\nolimits}
\def\tr{\mathop{\rm tr}\nolimits}

\def\lm{\lambda}

\def\C{\mathbb C}

\def\Z{\mathbb Z}

\def\Fold{\widetilde\Seq}
\def\Seq{\mathbf{Seq}}
\def\Top{\mathbf{Top}}
\def\GrMod{\mathbf{GrMod}}
\def\BS{\mathrm{BS}}
\def\Obj{\mathop{Obj}}
\def\Mor{\mathop{Mor}}

\def\<{\langle}
\def\>{\rangle}

\def\s{\mathbf s}

\renewcommand\phi{\varphi}

\def\bbeta{{\boldsymbol\beta}}
\def\im{\mathop{\rm im}}
\def\id{\mathrm{id}}
\renewcommand\epsilon{\varepsilon}
\def\Map{\mathop{\rm Map}}
\def\X{\mathcal X}
\def\suchthat{\mathbin{\rm |}}
\def\tH{\widetilde H}
\def\f{\mathbf f}
\def\ito{\stackrel\sim\to}
\def\pt{{\rm pt}}

\def\SL{\mathop{\rm SL}}
\def\PSL{\mathop{\rm PSL}}
\def\Stab{\mathop{\rm Stab\hspace{0.5pt}}\nolimits}

\newtheorem{theorem}{Theorem}[subsection]
\newtheorem{proposition}[theorem]{Proposition}
\newtheorem{lemma}[theorem]{Lemma}
\newtheorem{example}[theorem]{Example}
\newtheorem{corollary}[theorem]{Corollary}
\newtheorem{remark}[theorem]{Remark}

\newtheorem{definition}[theorem]{Definition}

\def\le{\leqslant}
\def\ge{\geqslant}

\renewcommand{\labelenumi}{{\rm\theenumi}}
\renewcommand{\theenumi}{{\rm(\arabic{enumi})}}

\def\={\equiv}
\def\cond#1{[#1]}
\def\2{{(2)}}

\def\ov#1#2{\genfrac{}{}{0pt}{}{#1}{#2}}

\title{Categories of Bott-Samelson varieties}
\author{Vladimir Shchigolev}
\thanks{Supported by the Russian Foundation for Basic Research grant no. 16-01-00756}
\email{shchigolev\_vladimir@yahoo.com}
\address{Financial University under the Government of the Russian Federation, 49 Leningradsky Prospekt, Moscow, Russia}

\begin{document}

\maketitle

\begin{abstract} We consider all Bott-Samelson varieties $\BS(s)$ for a fixed connected semisimple complex algebraic group with maximal torus $T$
as the class of objects of some category. The class of morphisms of this category is an extension of the class of canonical (inserting the neutral element)
morphisms $\BS(s)\hookrightarrow\BS(s')$, where $s$ is a subsequence of $s'$. Every morphism of the new category induces
a map between the $T$-fixed points but not necessarily between the whole varieties.

We construct a contravariant functor from this new category to
the category of graded $H^\bullet_T(\pt)$-modules coinciding on the objects with the usual functor $H_T^\bullet$
of taking $T$-equi\-variant cohomologies.

We also discuss the problem how to define a functor to the category of $T$-spaces
from a smaller subcategory. The exact answer is obtained for groups whose root systems have
simply laced irreducible components
by explicitly constructing morphisms between Bott-Samelson varieties (different from the canonical ones).
\end{abstract}


\section{Introduction}

Let $G$ be a connected semisimple complex algebraic group. Let $W$ be its Weyl group, $B$ be its Borel subgroup,
$T\subset B$ be its maximal torus and $\Pi$ be the set of simple reflections corresponding to $B$.
In this paper, we consider all complex varieties in the {\it metric topology}.
For any sequence
of simple reflections $s=(s_1,\ldots,s_r)$, we consider the {\it Bott-Samelson variety}
\begin{equation}\label{eq:BS}
\BS(s)=P_1\times\cdots\times P_r/B^r,
\end{equation}
where $P_i=B\cup Bs_iB$ is the parabolic subgroup for $s_i$ and $B^r$ acts as follows:
$$
(g_1,g_2,\ldots,g_r)\cdot(b_1,b_2,\ldots,b_r)=(g_1b_1,b_1^{-1}g_2b_2,\ldots,b_{r-1}^{-1}g_rb_r).
$$
We denote by $[g_1,\ldots,g_r]$ the point of $\BS(s)$ corresponding to the $r$-tuple $(g_1,\ldots,g_r)$.
It is well known that $\BS(s)$ is a smooth complex variety of dimension $r$.
As $T$ acts continuously on $\BS(s)$ via the first factor, we can consider the $T$-equivariant cohomology $H_T^\bullet(\BS(s))$.
This is a graded module over the graded ring $S=H_T^\bullet(\pt)$, where $\pt$ is the one-point \linebreak $T$-space
(we identify all one-point spaces and thus the corresponding algebras $S$).
The coefficient ring of these cohomologies is always a principal ideal domain with invertible~$2$
if the root system contains a component of type $C_n$.

The aim of this paper is to consider all Bott-Samelson varieties $\BS(s)$ For a fixed $G$ as a category.
This means that the Bott-Samelson varieties (or more exactly, their indexing sequences $s$) are the objects
of such a category and we need only to define the morphisms.

The simplest definition comes from the fact that $\BS(s)$ is a subvariety of $\BS(s')$,
whenever $s$ is a subsequence of $s'$. More exactly, let $\Seq$ be the category, whose objects are finite sequences of simple
reflections of $G$ (including the empty one) and any morphism $p:(s_1,\ldots,s_r)\to(s'_1,\ldots,s'_{r'})$ is a monotone
embedding $p:\{1,\ldots,r\}\to\{1,\ldots,r'\}$ such that $s'_{p(i)}=s_i$ for any $i$. The composition of such morphisms is the composition
in the category of sets.

We consider the covariant functor $\BS:\Seq\to \Top(T)$, where $\Top(T)$ is the category
of $T$-spaces (topological spaces with a continuous action of $T$).
This functor is already defined on objects by~(\ref{eq:BS}). 
For a morphism $p:s\to s'$, we define the canonical embedding
$$
\BS(p)([g_1,\ldots,g_r])=[g'_1,\ldots,g'_{r'}],\text{ where }
g'_j=
\left\{
\!\!
\begin{array}{ll}
g_i&\text{ if }j=p(i);\\[3pt]
e&\text{ if }j\notin\im p.
\end{array}
\right.
$$
It is natural to ask why we define $g'_j=e$ if $j\notin\im p$ and not $g'_j=s'_j$. The answer is that in the latter case,
we can not guarantee that $\BS(p)$ is a well-defined map.

In this paper, we show how to
eliminate this asymmetry by extending the set $\Mor(\Seq)$ to obtain a new category $\Fold$ with the same objects as $\Seq$.
In other words, $\Seq$ is a wide subcategory of $\Fold$. Every morphism $f:s\to s'$ of $\Fold$ induces a map
$f^T:\BS(s)^T\to\BS(s')^T$ on the $T$-fixed points (Section~\ref{morph}), although $f^T$ can not always be lifted to a continuous $T$-equi\-variant map
from $\BS(s)$ to $\BS(s')$. The maps $f^T$ for $f\in\Mor(\Fold)$ give a much larger (and more symmetric) class of maps between
the $T$-fixed point of Bott-Samelson varieties than the canonical maps $\BS(p)^T$ for $p\in\Mor(\Seq)$. 

The central idea of this paper is that in spite of the absence of a functor from $\Fold$ to $\Top(T)$, we have a
functor $\widetilde H$ from $\Fold$ to the category of graded $S$-modules $\GrMod(S)$. Thus we get the commutative diagram
\begin{equation}\label{eq:cdf}
\begin{tikzcd}
\Seq\arrow{r}{\BS}\arrow[hook]{rd}&\Top(T)\arrow{r}{H^\bullet_T}&\GrMod(S)\\
&\widetilde\Seq\arrow{ru}[swap]{\widetilde H}&
\end{tikzcd}
\end{equation}

We call $\Fold$ the {\it folding category} of $G$ for the following reason. Let $f\,{:}\,s\,{\to}\,s'$ be a mor\-phism
of $\Fold$ and $f^T:\BS(s)^T\to\BS(s')^T$ be the induced map on the $T$-fixed points.
These points can be interpreted as combinatorial galleries of Weyl chambers. If we know $f^T(\gamma)$
for one gallery $\gamma$, then we can reconstruct all other values $f^T(\gamma')$ in the following way.
Let $\gamma_1=\gamma,\gamma_2,\ldots,\gamma_n=\gamma'$ be a sequence such that $\gamma_{i+1}$ is obtained
from $\gamma_i$ by one fold at place $k_i$. Consider the sequence $\delta_1=f^T(\gamma),\delta_2,\ldots,\delta_n$ such that
$\delta_{i+1}$ is obtained from $\delta_i$ by the fold at place $p(k_i)$, where $p$ is some monotone embedding
uniquely defined by $f$ (not depending on $\gamma$). Then we have $f^T(\gamma')=\delta_n$
and this gallery depends only on $\gamma'$, not on the path to it from $\gamma$.

Still a question remains: does a functor to $\Top(T)$ assuring the commutativity exists from some smaller extension of $\Seq$ than $\Fold$?
%
%
More precisely, we ask if there exists a covariant functor $\BS':\Seq'\to\Top(T)$,
where $\Seq'$ is a subcategory of $\Fold$ containing $\Seq$, such that
the following diagram is commutative:
\begin{equation}\label{eq:z}
\begin{tikzcd}[column sep=9em]
\Seq\arrow{r}{\BS}\arrow[hook]{rd}\arrow[hook]{rdd}&\Top(T)\arrow{r}{H^\bullet_T}&\GrMod(S)\\
&\Seq'\arrow{ru}{H'}\arrow[hook]{d}\arrow[dashed]{u}{\BS'}&\\
&\widetilde\Seq\arrow{ruu}[swap]{\widetilde H}&
\end{tikzcd}
\end{equation}
Note that in the special case $\Seq'=\Seq$, the functor $\BS'$ exists and is equal to $\BS$.
The complete answer is obtained in the case of root systems with simply laced irreducible components (Theorem~\ref{theorem:cat}).

In Section~\ref{Generalities}, we set the basic notation and define the main characters
of this paper (Bott-Samelson varieties, combinatorial galleries and so on). We also discus
the coordinatization of $G$ as a Chevalley group and the cohomological criterion due to
M.\,H\"arterich~\cite{Haerterich}.

We define the folding category $\Fold$ in Section~\ref{fcFftH}. The main result here
is Theorem~\ref{theorem:1:rocaam}. By means of the functor $\widetilde H$ constructed in Section~\ref{widetildeH},
this theorem allows us to restrict \linebreak $T$-equivariant cohomologies
of Bott-Samelson varieties along morphisms of $\Fold$ in the same way as we
restrict cohomologies along usual morphisms.

In Section~\ref{Topology}, we describe those morphisms of $\Fold$ that do induce $T$-equivariant
continuous maps between Bott-Samelson varieties. We call these morphisms topological
(see Definition~\ref{def:1} and Theorem~\ref{theorem:1} for a combinatorial
criterion\footnote{We obtain even $T$-equivariant morphisms of algebraic varieties in this way.}).
It turns out that the question about the existence of a functor $H'$ making Diagram~(\ref{eq:z})
commutative is answered (for simply laced root systems) in terms of topological morphisms
(Theorem~\ref{theorem:cat}). It is easy to observe that topological morphisms are quite rare
(Examples~\ref{Ex:3}--\ref{Ex:6} and~\ref{Ex:8}). Therefore it would be interesting to find
a topological explanation for the existence of the functor $\widetilde H$.

Our technique of the topological part is based on the criterion for existence of $T$-curves
between certain points of Bott-Samelson varieties (Lemma~\ref{lemma:fcat:6.5}). To prove this result and
to lift topological morphisms to $T$-equivariant continuous maps between whole Bott-Samelson
varieties in Theorem~\ref{theorem:1}, we need to be able to calculate transition
functions between adjacent charts (Lemma~\ref{lemma:m}). We use the general idea (Proposition~\ref{lemma:1})
that the transition functions between any charts of Bott-Samelson varieties are calculated
inductively via certain sequences of elements of the Borel subgroup,
which we call transition sequences. It is controlling this sequences that allows us to produce
manageable transition functions.

We conclude the paper with Appendix. It consists of numerical examples and the proof
of one phenomenon (Corollary~\ref{corollary:2}), which is not used elsewhere in the paper.

\section{Generalities}\label{Generalities}

\subsection{Set theoretic notation} We denote the fact that $N$ is a subset of $M$, including the case $N=M$, by $N\subset M$,
reserving the notation $N\varsubsetneq M$ for the proper inclusion. We denote the cardinality of a finite set $M$ by $|M|$.
If $s$ is a sequence, then we write $|s|$ for its length and
$s_i$ for its $i$th entry. So we have $s=(s_1,\ldots,s_{|s|})$.
The identity map from a set $S$ to itself is denoted by $\id_S$.
The set of all maps from a set $X$ to a set $Y$ is denoted by $\Map(X,Y)$.

All indices are assumed integral and $[a,b]=\{a,a+1,\ldots,b\}$, $[a,b)=\{a,a+1,\ldots,b-1\}$ an so on
for any integers $a$ and $b$. For any
logical condition $\mathcal P$, we set $[\mathcal P]=1$ if $\mathcal P$ is true and $[\mathcal P]=0$ if not.
We use both notations $\phi\psi$ and $\phi\circ\psi$ for the composition.

\subsection{Root system} We fix a root system $R$ spanning its Euclidian space $E$ with product $(?,?)$.
Let $\Pi\subset R^+$ be sets of simple and positive roots respectively. The hyperplanes $L_\alpha=\{v\in E\suchthat (v,\alpha)=0\}$,
where $\alpha\in R$, define the chambers in E as the connected components of the difference $E\setminus\bigcup_{\alpha\in R}L_\alpha$.
The chambers are transitively permuted by the Weyl group~$W$. So each chamber is uniquely written as
$wC$ for $w\in W$, where $C=\{v\in E\suchthat (v,\alpha)>0\}$ is the fundamental chamber.

\begin{proposition}\label{proposition:1} Let $u,v\in W$ and $\alpha\in R$. The hyperplane $L_\alpha$ does not separate the chambers $uC$ and $vC$
if and only if the roots $u^{-1}\alpha$ and $v^{-1}\alpha$ are both positive or both negative. 
\end{proposition}

For any root $\alpha\in R$, we denote the reflection of $E$ with respect to $L_\alpha$ by $s_\alpha$.
It is well known that $W$ is a Coxeter group with generators $\{s_\alpha\suchthat\alpha\in\Pi\}$.
We have $wL_\alpha=L_{w\alpha}$ and $ws_{\alpha}w^{-1}=s_{w\alpha}$ for any $w\in W$ and $\alpha\in R$.

\subsection{Chevalley groups}${}$\!\!\! We will consider the group $G$ constructed as a Chevalley group~\cite{Steinberg}.
It is generated by the root elements $x_\alpha(c)$, where $\alpha\in R$ and $c\in\C$.
We fix the following important elements:
$$
s_\alpha(c)=x_{\alpha}(c)x_{-\alpha}(-c^{-1})x_{\alpha}(c),\quad s_\alpha=s_\alpha(1),\quad h_\alpha(c)=s_\alpha(c)s_\alpha^{-1}.
$$
These elements are subject to the following relations:
\smallskip
{
\renewcommand{\labelenumi}{{\rm \theenumi}}
\renewcommand{\theenumi}{{\rm(R\arabic{enumi})}}
\begin{enumerate}\itemsep=10pt
\item\label{rel:1} $x_\alpha(c)x_\alpha(d)=x_\alpha(c+d)$;
\item\label{rel:2} $\displaystyle x_\alpha(c)x_\beta(d)x_\alpha(c)^{-1}x_\beta(d)^{-1}=\prod_{i,j>0}x_{i\alpha+j\beta}(\xi_{i,j}c^id^j)$ if $\alpha+\beta\ne0$;\\[-16pt]
\item\label{rel:3} $s_\alpha h_\beta(c) s_\alpha^{-1}=h_{s_\alpha\beta}(c)$;
\item\label{rel:4} $s_\alpha x_\beta(t)s_\alpha^{-1}=x_{s_{\alpha}\beta}(\sigma_{\alpha,\beta}t)$;
\item\label{rel:5} $h_\alpha(c)x_\beta(d)h_\alpha(c)^{-1}=x_\beta(c^{\<\beta,\alpha\>}d)$.
\end{enumerate}}
\medskip

\noindent
The product in~\ref{rel:2} is taken only over $i$ and $j$ such that $i\alpha+j\beta$ is a root and
the coefficients $\xi_{i,j}$ do not depend on $c$ and $d$. In~\ref{rel:4}, we have $\sigma_{\alpha,\beta}=\pm1$.
We note the following important property: $\sigma_{\alpha,\beta}=\sigma_{\alpha,-\beta}$.

The maximal torus $T$ is the subgroup of $G$ generated by all $h_\alpha(c)$ and the Borel subgroup $B$ is the subgroup of $G$
generated by $T$ and all $x_\alpha(c)$ with $\alpha>0$. The Weyl group of our root system $W$ can be identified with
the quotient $N/T$, where $N$ is the subgroup of $G$ generated by all $s_\alpha(c)$.
We will often deal with the flag variety $G/B$.
Each its $T$-fixed point can be written uniquely as $wB$ for some $w\in W$. So,
abusing notation, we will denote this point simply by $w$.


For any root $\alpha\in R$, there exists a group homomorphism from $\SL_2(\C)$ to $G$ such that
\begin{equation}\label{eq:SL:PSL}
\left(\!
\begin{array}{cc}
1&c\\
0&1
\end{array}\!
\right)\mapsto x_\alpha(c),
\qquad
\left(\!
\begin{array}{cc}
1&0\\
c&1
\end{array}\!
\right)\mapsto x_{-\alpha}(c).
\end{equation}
Thus calculations in $\SL_2(\C)$ imply that
\begin{equation}\label{eq:comm:1}
x_\alpha(c)x_{-\alpha}(d)=x_{-\alpha}\(\tfrac d{cd+1}\)x_\alpha\big(c(cd+1)\big)h_\alpha(cd+1)
\end{equation}
if $cd\ne-1$ and that
\begin{equation}\label{eq:comm:2}
x_\alpha(c)x_{-\alpha}\(-\tfrac1c\)=s_\alpha x_\alpha\(-\tfrac1c\)h_\alpha\(\tfrac1c\).
\end{equation}
if $c\ne0$.

We call a root system {\it simply laced\label{simply_laced}} if all its irreducible components
are simply laced (that is of types $A_n$, $D_n$ or $E_n$). In that case, for any two roots $\alpha$ and $\beta$ with nonzero sum, we have the following commutation formula:
\begin{equation}\label{eq:comm}
x_\alpha(c)x_\beta(d)=x_\beta(d)x_\alpha(c)x_{s_\beta\alpha}(\epsilon cd)
\end{equation}
for some $\epsilon\in\{-1,0,1\}$, where $s_\beta\alpha=\alpha+\beta$ if $\epsilon\ne0$.

\begin{remark} \rm The reader should be cautious when dealing with $s_\alpha$. It denotes either the element of the Weyl group $W$
or the element of $G$ depending on the context. We will indicate which of two we mean if ambiguity
arises.

\end{remark}

\subsection{Bott-Samelson varieties}\label{BS} The Bott-Samelson variety $\BS(s)$ for a sequence of simple roots $s=(s_1,\ldots,s_r)$ is defined by
formula~(\ref{eq:BS}) of the introduction. Every point of this variety can be written as $[g_1,\ldots,g_r]$
with the notation also introduced there. The maximal torus $T$ acts on $\BS(s)$ via the first factor: $t\cdot[g_1,\ldots,g_r]=[tg_1,\ldots,g_r]$.

For $i=0,\ldots,r$, we have the
truncation map $\tr_i:\BS(s)\to\BS(s')$, where $s'=(s_1,\ldots,s_i)$, and the projection map $\pi_i:\BS(s)\to G/B$ defined as follows:
$$
\tr_i([g_1,\ldots,g_r])=[g_1,\ldots,g_i],\qquad \pi_i([g_1,\ldots,g_r])=g_1\cdots g_iB/B.
$$
Note that $\BS(s)$ can be considered as a subvariety of $(G/B)^r$ via the embedding
\begin{equation}\label{eq:a}
\iota:\BS(s)\hookrightarrow(G/B)^r,\qquad \iota(x)=(\pi_1(x),\pi_2(x),\ldots,\pi_r(x)).
\end{equation}
All the above maps are obviously $T$-equivariant.

The set of the $T$-fixed points $\BS(s)^T$ can be identified\footnote{$[\gamma_1,\ldots,\gamma_r]$ is identified with $(\gamma_1,\ldots,\gamma_r)$.} with the set
$$
\Gamma_s=\{(\gamma_1,\ldots,\gamma_r)\suchthat\gamma_i=s_i\text{ or }\gamma_i=e\}.
$$
Its elements are called {\it combinatorial galleries}\label{cg}. The gallery of Weyl chambers corresponding
to $\gamma=(\gamma_1,\ldots,\gamma_r)\in\Gamma_s$ is $(\gamma^0C,\gamma^1C,\gamma^2C,\ldots,\gamma^rC)$,
where $\gamma^i=\gamma_1\gamma_2\ldots\gamma_i$. The common wall of the adjacent chambers $\gamma^{i-1}C$ and $\gamma^iC$
is $L_{\bbeta_i(\gamma)}$, where $\bbeta_i(\gamma)=\gamma^i(-\alpha_i)$, $s_i=s_{\alpha_i}$ and $\alpha_i$ is a simple root.
Proposition~\ref{proposition:1} implies that $L_{\bbeta_i(\gamma)}$ separates $\gamma^iC$ and $C$ if and only if
$\bbeta_i(\gamma)>0$. Hence we get the following useful criterion.

\begin{proposition}\label{proposition:2}
Let $\gamma\in\Gamma_s$, $\delta\in\Gamma_{s'}$, $i=1,\ldots,|s|$, $j=1,\ldots,|s'|$ and $\beta\in R$ be such that
$\bbeta_i(\gamma)=\pm\beta$ and $\bbeta_j(\delta)=\pm\beta$. The hyperplane $L_\beta$ does not separate the chambers
$\gamma^iC$ and $\delta^jC$ if and only if $\bbeta_i(\gamma)=\bbeta_j(\delta)$.
\end{proposition}

The Bott-Samelson variety $\BS(s)$ is covered by Zariski open affine subvarieties (charts)
$$
U^\gamma_s=\{[x_{\gamma_1(-\alpha_1)}(c_1)\gamma_1,\ldots,x_{\gamma_r(-\alpha_r)}(c_r)\gamma_r]\suchthat c_1,\ldots,c_r\in\C\}
$$
Note that in this formula, $\gamma_i$ in the subscript denotes an element of $W$ and $\gamma_i$ in the line denotes an element
of $G$ ($e$ or $s_{\alpha_i}(1)$). For brevity, we set
$${\arraycolsep=2pt
\begin{array}{ccl}
[c_1,\ldots,c_r]^\gamma_s&=&[c]^\gamma_s=[x_{\gamma_1(-\alpha_1)}(c_1)\gamma_1,\ldots,x_{\gamma_r(-\alpha_r)}(c_r)\gamma_r],\\[7pt]
\llbracket c_1,\ldots,c_i\rrbracket^\gamma_s&=&x_{\gamma_1(-\alpha_1)}(c_1)\gamma_1\cdots x_{\gamma_i(-\alpha_i)}(c_i)\gamma_i.
\end{array}}
$$
We have $U_s^\gamma\cong\C^r$ with the following action of $T$:
$$
t\cdot[c_1,\ldots,c_r]_s^\gamma=[\bbeta_1(\gamma)(t)c_1,\ldots,\bbeta_r(\gamma)(t)c_r]_s^\gamma.
$$
We need the elements $\llbracket c_1,\ldots,c_i\rrbracket^\gamma_s$ to calculate the so-called transition sequence
(responsible for the transition functions between charts) in Proposition~\ref{lemma:1} and at other instances
where this proposition is applied. Therefore we study them in more detail.
For $\gamma\in\{e,s_\alpha\}$ with $\alpha\in\Pi$, we set $\sigma^\gamma_\beta=\sigma_{\alpha,\beta}$ if $\gamma=s_\alpha$
and $\sigma^\gamma_\beta=1$ if $\gamma=e$. If $\gamma\in\Gamma_s$, then we set
$$
\sigma_\beta^{\gamma,i}=\sigma_{\gamma_2\cdots\gamma_i(\beta)}^{\gamma_1}\sigma_{\gamma_3\cdots\gamma_i(\beta)}^{\gamma_2}
\cdots
\sigma_{\gamma_{i-1}\gamma_i(\beta)}^{\gamma_{i-2}}
\sigma_{\gamma_i(\beta)}^{\gamma_{i-1}}.
$$
Clearly, $\sigma_{-\beta}^{\gamma,i}=\sigma_\beta^{\gamma,i}$. Thus
$$
\llbracket c_1,\ldots,c_i\rrbracket_s^\gamma =
x_{\bbeta_1(\gamma)}(\sigma^{\gamma,1}_{\alpha_1}c_1)
x_{\bbeta_2(\gamma)}(\sigma^{\gamma,2}_{\alpha_2}c_2)\cdots
x_{\bbeta_i(\gamma)}(\sigma^{\gamma,i}_{\alpha_i}c_i)\gamma^i,
$$
where $\gamma\in\Gamma_s$ and $i=0,\ldots,r$.

For any root $\alpha\in R$, we define on $\Gamma_s$ the following equivalence relation:
$$
\gamma\sim_\alpha\delta\Longleftrightarrow\gamma_i=\delta_i\text{ unless }\bbeta_i(\gamma)=\pm\alpha
$$
and the following total order:
$$
\delta\vartriangleleft\gamma\Longleftrightarrow\exists k:\;\delta^k<\gamma^k\text{ (in the Bruhat order)\; and \;}\forall i<k:\;\delta^i=\gamma^i.
$$
Moreover, we define the following two sets:
$$
M_\alpha(\gamma)=\{i=1,\ldots,r\suchthat\bbeta_i(\gamma)=\pm\alpha\},\qquad
J_\alpha(\gamma)=\{i=1,\ldots,r\suchthat\bbeta_i(\gamma)=\alpha\}.
$$

Finally, we introduce the {\it folding operator}\label{fo} $\f_i$ that folds every gallery at the $i$th place.
In other words,
$$
(\f_i\gamma)_j=\left\{
\begin{array}{ll}
\gamma_is_i&\text{ if }j=i;\\
\gamma_j&\text{ otherwise }.
\end{array}
\right.
$$

It is easy to prove that

\begin{equation}\label{eq:-2}
(\f_i\gamma)^k=s_{\bbeta_i(\gamma)}^{\cond{i\le k}}\gamma^k,\qquad \bbeta_k(\f_i\gamma)=s_{\bbeta_i(\gamma)}^{\cond{i\le k}}\bbeta_k(\gamma).
\end{equation}

%

\subsection{Cohomology of Bott-Samelson varieties} Let $k$ be a principal
ideal domain with invertible $2$ if the root system contains a component of type $C_n$.
For any $T$-space $X$, we will consider its $T$-equivariant cohomology
$H^\bullet_T(X)=H^\bullet_T(X,k)$ with the ring of coefficients~$k$.
For our choice of $k$, the restriction map $H^\bullet_T(\BS(s))\to H^\bullet_T(\Gamma_s)$
is an embedding. We denote the image of this embedding by $\X(s)$ or $\X(s,k)$ if
we need to underline the ring of coefficients $k$. Similarly, we denote
the cohomology ring $H^\bullet_T(\pt,k)$ of the point by $S$ or more precisely by $S_k$.
It is known that $S$ is isomorphic to the symmetric algebra of the space $X(T)\otimes_\Z k$,
where $X(T)$ is the group of characters of $T$. Therefore $S$ is a commutative ring
with the action of $W$. Moreover, $H^\bullet_T(\BS(s))$,
$H^\bullet_T(\Gamma_s)$ and $\X(s)$ are $S$-modules.
If $X$ is a finite $T$-space with discrete topology (and thus the identical action of $T$), then
we identify $H^\bullet_T(X)$ with $\Map(X,S)$ by the equivariant Mayer-Vietoris sequence.

We will use the following criterion due to
M.\,H\"arterich\footnote{This is the formulation of~\cite{scbs} with the choice of coefficients $\Z'$
appropriate for the intersection formula~\cite[Corollary~2.5]{scbs} in the localization theorem.}.

\begin{proposition}[\mbox{\cite[Theorem 6.2]{Haerterich}}]\label{proposition:x:1}
Let $\Z'=\Z[1/2]$ if the root system $R$ contains a component of type $C_n$
and $\Z'=\Z$ otherwise.
An element $f\in H_T^\bullet(\Gamma_s,\Z')$ belongs to $\X(s,\Z')$  if and only if
$$
\sum_{\delta\in\Gamma_s,\delta\sim_\alpha\gamma,J_\alpha(\delta)\subset J_\alpha(\gamma)}(-1)^{|J_\alpha(\delta)|}f(\delta)\=0\pmod{\alpha^{|J_\alpha(\gamma)|}}
$$
for any positive root $\alpha$ and gallery $\gamma\in\Gamma_s$.
\end{proposition}

We have the following commutative diagram of morphisms of $S$-modules:
\begin{equation}\label{eq:-3}
\begin{tikzcd}
\X(s,\Z')\otimes_{S_{\Z'}}S\arrow{r}{\sim}\arrow[hook]{d}&\X(s)\arrow[hook]{d}\\
H^\bullet_T(\Gamma_s,\Z')\otimes_{S_{\Z'}} S\arrow{r}{\zeta}[swap]{\sim}&H_T^\bullet(\Gamma_s)
\end{tikzcd}
\end{equation}
The natural isomorphism $\zeta$ in the bottom line acts as follows
$\zeta(f\otimes\lm)(\gamma)=\zeta'(f(\gamma))\lm$, where $\zeta':S_{\Z'}\to S$ the ring homomorphism
coming from the natural ring homomorphism $\Z'\to k$.


\section{\texorpdfstring{Category $\Fold$ and Functor $\widetilde H$}{Category tilde Seq and Functor tilde H}}\label{fcFftH}

\subsection{Case of rank 2}\label{rank2}

Let $\alpha\in R$. We denote by $G_\alpha$ the subgroup of $G$ generated by all root elements
$x_\alpha(c)$ and $x_{-\alpha}(c)$. It is well known that $G_\alpha$
is isomorphic to either $\SL_2(\C)$ or $\PSL_2(\C)$, the isomorphism being given by~(\ref{eq:SL:PSL}).
The subgroup $G_\alpha$ contains its own maximal torus $T_\alpha=\{h_\alpha(c)\suchthat c\in\C\}$ and
Borel subgroup $B_\alpha$ generated by $T_\alpha$ and all root elements $x_\alpha(c)$.
Therefore for every integer $\l\ge0$, we can consider the Bott-Samelson
variety
$$
\BS^\2_{\alpha,\l}=\BS((\underbrace{s_\alpha,\ldots,s_\alpha}_{\l\text{ times.}})).
$$
for the group $G_\alpha$. The embedding of $\BS^\2_{\alpha,\l}$ into $(G_\alpha/B_\alpha)^\l$ given by~(\ref{eq:a})
in the present case is an isomorphism and we denote it by $\iota_2$ to avoid confusion in the proof
of Theorem~\ref{theorem:1:rocaam}.

We note that the `big' torus $T$ acts on $G_\alpha/B_\alpha$ and $\BS^\2_{\alpha,\l}$ by
$$
t\cdot gB_\alpha=tgt^{-1}B_\alpha,\qquad t\cdot[g_1,\ldots,g_\l]=[tg_1t^{-1},\ldots,tg_\l t^{-1}].
$$
The isomorphism $\iota_2:\BS^\2_{\alpha,\l}\ito(G_\alpha/B_\alpha)^\l$ is obviously $T$-equivariant.

We denote by $\Gamma^\2_{\alpha,\l}$ the set of $T_\alpha$-fixed points of $\BS^\2_{\alpha,\l}$
and by $\X^\2_{\alpha,\l}$ the image of the restriction
$H^\bullet_T(\BS^\2_{\alpha,\l})\to H^\bullet_T(\Gamma^\2_{\alpha,\l})$.
Note that $\Gamma^\2_{\alpha,\l}$ is also  the set of $T$-fixed points of $\BS^\2_{\alpha,\l}$.
We denote $s=s_\alpha$ for brevity ($\alpha$ is the unique simple root).

Let $K\subset[1,\l]$ and $L\subset[1,\l]$ be disjoint sets and $k:K\to\{e,s\}$ be an arbitrary function.
We consider the following subset of $\BS^\2_{\alpha,\l}$:
$$
\Sigma_{K,L}^k=\iota_2^{-1}\big(\{(g_1,\ldots,g_\l)\in(G_\alpha/B_\alpha)^\l\suchthat \forall i\in K: g_i=k_i;\;\forall i\in L: g_{i-1}=g_i\}\big),
$$
where $g_0=e$. This subset is clearly closed and $T$-equivariant.
Consider the subset $[1,\l]\setminus(K\cup L)=\{i_1<\cdots<i_m\}$ and let $\phi:(G_\alpha/B_\alpha)^\l\to(G_\alpha/B_\alpha)^m$ be the projection to
this subset: $\phi(g_1,\ldots,g_\l)=(g'_1,\ldots,g'_m)$, where $g'_j=g_{i_j}$.
We have the chain of isomorphisms
$$
\begin{tikzcd}
\Sigma_{K,L}^k\arrow{r}{\iota_2}[swap]{\sim}&\iota_2(\Sigma_{K,L}^k)\arrow{r}{\phi}[swap]{\sim}&(G_\alpha/B_\alpha)^m
\end{tikzcd}
$$
Hence the $T$-equivariant Euler class is equal to $\Eu_T(\delta,\Sigma^k_{K,L})=\prod_{i\in[1,\l]\setminus(K\cup L)}\delta^i(-\alpha)$.
Considering the commutative diagram of restrictions
$$
\begin{tikzcd}
H_T^\bullet(\BS^\2_{\alpha,\l})\arrow{r}\arrow{d}&H_T^\bullet(\Sigma_{K,L}^k)\arrow{d}&\\
H_T^\bullet(\Gamma^\2_{\alpha,\l})\arrow{r}& H_T^\bullet((\Sigma_{K,L}^k)^T)\arrow[equal]{r}&H_T^\bullet(\{\delta\in\Gamma^\2_{\alpha,\l}\suchthat \forall i\in K:\delta^i=k_i;\;\forall i\in L:\delta_i=e\})
\end{tikzcd}
$$
we get
\begin{equation}\label{eq:11}
\sum_{\ov{\forall i\in K:\mu^i=k_i}{\forall i\in L:\mu_i=e}}\frac{f(\mu)}{\prod_{i\in[1,\l]\setminus(K\cup L)}\mu^i(-\alpha)}\in S
\end{equation}
for any $f\in\X^\2_{\alpha,\l}$ (see the arguments in~\cite[Section 5]{Haerterich} or~\cite{Atiyah_Bott}).
We are going to generalize this fact as follows.





\begin{lemma}\label{lemma:y}
Let $f\in\X^\2_{\alpha,\l}$ and $k:K\to\{e,s\}$, $l:L\to\{e,s\}$ be arbitrary functions. Then
\begin{equation}\label{eq:14}
\sum_{\ov{\forall i\in K:\mu^i=k_i}{\forall i\in L:\mu_i=l_i}}\frac{f(\mu)}{\prod_{i\in[1,\l]\setminus(K\cup L)}\mu^i(-\alpha)}\in S.
\end{equation}
\end{lemma}
\begin{proof}
Let $|l|$ denote the number of indices $p\in L$ such that $l_p=s$. We are going to prove this lemma by induction on $|l|$.
The case $|l|=0$ is dealt with by~(\ref{eq:11}).

Now suppose that $l_p=s$ for some $p\in L$. We denote $q=\max\,[0,p)\setminus L$. We have $(q,p]\subset L$.
First consider the case $q\in K\cup\{0\}$.
The condition $\mu_p=l_p=s$ is equivalent to $\mu^p=k_ql_{q+1}\cdots l_p$,
where $k_0=e$. Thus~(\ref{eq:14}) can be rewritten as follows:
$$
\sum_{\ov{\forall i\in K\cup\{p\}: \mu^i=k_i}{\forall i\in L\setminus\{p\}:\mu_i=l_i}}\frac{f(\mu)}{\prod_{i\in[1,\l]\setminus(K\cup L)}\mu^i(-\alpha)},
$$
where $k_p=k_ql_{q+1}\cdots l_p$. By the inductive hypothesis, this sum belongs to $S$.

Finally, consider the case $q>0$ and $q\notin K$. 
We define the functions $k':K\cup\{p\}\to\{e,s\}$ and $k'':K\cup\{q\}\to\{e,s\}$
by $k'_i=k''_i=k_i$ for $i\in K$, $k'_p=e$ and $k''_q=l_{q+1}\cdots l_{p-1}$.
The inductive hypothesis applied to the pair of functions $l:L\setminus\{p\}\to\{e,s\}$ and \linebreak $k':K\cup\{p\}\to\{e,s\}$ implies
\begin{equation}\label{eq:y}
\begin{array}{l}
\hspace{-30pt}
\displaystyle\sum_{\ov{\forall i\in K: \mu^i=k_i,\,\mu^q=l_{q+1}\cdots l_{p-1}}{\forall i\in L\setminus\{p\}: \mu_i=l_i,\,\mu_p=e}}\frac{f(\mu)}{\mu^q(-\alpha)\prod_{i\in[1,\l]\setminus(K\cup L\cup\{q\})}\mu^i(-\alpha)}\\[30pt]
\hspace{90pt}
\displaystyle+\sum_{\ov{\forall i\in K: \mu^i=k_i,\,\mu^q=l_{q+1}\cdots l_{p-1}s}{\forall i\in L\setminus\{p\}: \mu_i=l_i,\,\mu_p=s}}\frac{f(\mu)}{\mu^q(-\alpha)\prod_{i\in[1,\l]\setminus(K\cup L\cup\{q\})}\mu^i(-\alpha)}\in S
\end{array}
\end{equation}
and applied to the pair of functions $l:L\setminus\{p\}\to\{e,s\}$ и $k'':K\cup\{q\}\to\{e,s\}$ implies
\begin{equation}\label{eq:yy}
\begin{array}{l}
\hspace{-30pt}
\displaystyle\sum_{\ov{\forall i\in K:\mu^i=k_i,\,\mu^q=l_{q+1}\cdots l_{p-1}}{\forall i\in L\setminus\{p\}:\mu_i=l_i\,,\,\mu_p=e}}\frac{f(\mu)}{\mu^p(-\alpha)\prod_{i\in[1,\l]\setminus(K\cup L\cup\{q\})}\mu^i(-\alpha)}\\[30pt]
\hspace{90pt}
\displaystyle+\sum_{\ov{\forall i\in K:\mu^i=k_i,\,\mu^q=l_{q+1}\cdots l_{p-1}}{\forall i\in L\setminus\{p\}:\mu_i=l_i\,,\,\mu_p=s}}\frac{f(\mu)}{\mu^p(-\alpha)\prod_{i\in[1,\l]\setminus(K\cup L\cup\{q\})}\mu^i(-\alpha)}\in S.
\end{array}
\end{equation}
In the first summands of both formulas~(\ref{eq:y}) and~(\ref{eq:yy}), we have
$$
\mu^p(-\alpha)=\mu^q\mu_{q+1}\cdots\mu_{p-1}(-\alpha)=\mu^ql_{q+1}\cdots l_{p-1}(-\alpha)=\frac{l_{q+1}\cdots l_{p-1}(\alpha)}{\alpha}\mu^q(-\alpha).
$$
In the second summands of both formulas~(\ref{eq:y}) and~(\ref{eq:yy}), we have
$$
\mu^p(-\alpha)=\mu^q\mu_{q+1}\cdots\mu_{p-1}s(-\alpha)=-\mu^ql_{q+1}\cdots l_{p-1}(-\alpha)=-\frac{l_{q+1}\cdots l_{p-1}(\alpha)}{\alpha}\mu^q(-\alpha).
$$
Therefore, if we multiply~(\ref{eq:yy}) by $-\alpha/l_{q+1}\cdots l_{p-1}(\alpha)$ (which equals $\pm1$)
and add it to~(\ref{eq:y}), then the first summands cancel out and we get
\begin{multline*}
\sum_{\ov{\forall i\in K: \mu^i=k_i,\,\mu^q=l_{q+1}\cdots l_{p-1}s}{\forall i\in L\setminus\{p\}: \mu_i=l_i,\,\mu_p=s}}\frac{f(\mu)}{\mu^q(-\alpha)\prod_{i\in[1,\l]\setminus(K\cup L\cup\{q\})}\mu^i(-\alpha)}\\
+
\sum_{\ov{\forall i\in K:\mu^i=k_i,\,\mu^q=l_{q+1}\cdots l_{p-1}}{\forall i\in L\setminus\{p\}:\mu_i=l_i\,,\,\mu_p=s}}\frac{f(\mu)}{\mu^q(-\alpha)\prod_{i\in[1,\l]\setminus(K\cup L\cup\{q\})}\mu^i(-\alpha)}
\in S.
\end{multline*}
It is obvious that the sum at the left-hand side is exactly the sum
from the formulation of the lemma (divided into two sums according to the value of $\mu^q$).
\end{proof}

\subsection{$(p,w)$-pairs}\label{pw} Let $s=(s_1,\ldots,s_r)$ and $s'=(s'_1,\ldots,s'_{r'})$ be sequences of simple reflections,
$p:[1,r]\to[1,r']$ be a monotone embedding and $w\in W$. A pair of galleries $(\gamma,\delta)\in\Gamma_s\times\Gamma_{s'}$
is called a {\it $(p,w)$-pair} if
$$
\delta^{p(i)}s'_{p(i)}(\delta^{p(i)})^{-1}=w\gamma^is_i(\gamma^i)^{-1}w^{-1}
$$
for all $i=1,\ldots,r$. We shall write these equations in the more convenient form
$$
\bbeta_{p(i)}(\delta)=\pm w\bbeta_i(\gamma).
$$

\begin{example}{\rm Let $s$ and $\gamma$ be beginnings of $s'$ and $\delta$ respectively having the same lengths.
Then $(\gamma,\delta)$ is an $(e,p)$-pair, where $p$ is the natural embedding $[1,|s|]\to[1,|s'|]$.}
\end{example}

\begin{example}{\rm Let $s$ and $\gamma$ be the ends of $s'$ and $\delta$ respectively of length $r\le|s'|$.
Then $(\gamma,\delta)$ is a $(\delta^{|s'|-r},p)$-pair, where $p(i)=|s'|-r+i$.}
\end{example}

\begin{lemma}\label{lemma:x2}${}$\!\!\!
Let $\bbeta_{p(k)}(\delta)=\epsilon_kw\bbeta_k(\gamma)$ for any $k=1,\ldots,|\gamma|$, where $\epsilon_k=\pm1$.
Then $\bbeta_{p(k)}(\f_{p(i)}\delta)=\epsilon_k w\bbeta_k(\f_i\gamma)$ for any $i,k=1,\ldots,|\gamma|$.
In particular, if $(\gamma,\delta)$ is a $(p,w)$-pair, then $(\f_i\gamma,\f_{p(i)}\delta)$ is also a $(p,w)$-pair.
\end{lemma}
\begin{proof} Applying~(\ref{eq:-2}) twice, we get
$$
\bbeta_{p(k)}(\f_{p(i)}\delta)=s_{\bbeta_{p(i)}(\delta)}^{\cond{p(i)\le p(k)}}\bbeta_{p(k)}(\delta)
=\epsilon_k s_{w\bbeta_i(\gamma)}^{\cond{i\le k}}w\bbeta_{k}(\gamma)=\epsilon_k ws_{\bbeta_i(\gamma)}^{\cond{i\le k}}\bbeta_{k}(\gamma)
=\epsilon_k w\bbeta_k(\f_i\gamma).
$$
\end{proof}


\subsection{\texorpdfstring{Category $\Fold$}{Category tilde Seq}}\label{morph}

The objects of this category are finite sequences of simple reflections
(including the empty one). A morphism from $s=(s_1,\ldots,s_r)$ to
$s'=(s'_1,\ldots,s'_{r'})$ is a triple $(p,w,\phi)$, where

\begin{enumerate}
\item\label{F:1} $p:[1,r]\to[1,r']$ is monotone embedding;
\item\label{F:2} $w\in W$;
\item\label{F:3} $\phi:\Gamma_s\to\Gamma_{s'}$ is a map such that $(\gamma,\phi(\gamma))$
                 is a $(p,w)$-pair and $\phi(\f_i\gamma)=\f_{p(i)}\phi(\gamma)$ for any
                 $\gamma\in\Gamma_s$ and $i=1,\ldots,r$.
\end{enumerate}
The first part of the last condition can be written in the following form:
$$
\bbeta_{p(i)}(\phi(\gamma))=\epsilon_i w\bbeta_i(\gamma)\quad\forall
i=1,\ldots,r,
$$
where $\epsilon_i=\pm1$. Lemma~\ref{lemma:x2} shows that the numbers $\epsilon_i$
do not depend on $\gamma$.
We call the sequence $(\epsilon_1,\ldots,\epsilon_r)$ and $w\in W$ respectively
the {\it sign}\label{signepsilon} and the {\it rotation}\label{rotw} of $(p,w,\phi)$.
The map between the $T$-fixed points is $(p,w,\phi)^T=\phi$.
We say that the sign is {\it positive}\label{possign} if $\epsilon_i=1$ for any $i$ and that the rotation is
{\it identical}\label{idrot} if $w=e$. Both these conditions are satisfied for
the identity morphism $(\id_{[1,r]},e,\id_{\Gamma_s})$ from $s$ to itself.

The composition of two morphisms is given by the natural formula
$$
(p',w',\phi')\circ(p,w,\phi)=(p'p,w'w,\phi'\phi).
$$
As the composition $p'p$ is monotone and $w'w\in W$, it remains
to check condition~\ref{F:3}. Let $(\epsilon_1,\ldots,\epsilon_{|s|})$ be the sign
of $(p,w,\phi)$ and $(\epsilon'_1,\ldots,\epsilon'_{|s'|})$ be the sign
of $(p',w',\phi')$. Then we have
$$
\bbeta_{p'p(i)}(\phi'\phi(\gamma))=\epsilon'_{p(i)}
w'\bbeta_{p(i)}(\phi(\gamma))=\epsilon'_{p(i)}\epsilon_iw'w\bbeta_i(\gamma),
$$
$$
\phi'\phi(\f_i(\gamma))=\phi'\big(\f_{p(i)}(\phi(\gamma))\big)=\f_{p'p(i)}(\phi'\phi(\gamma)).
$$
Hence the composition $(p',w',\phi')(p,w,\phi)$ is a morphism of sign
$(\epsilon'_{p(1)}\epsilon_1,\ldots,\epsilon'_{p(r)}\epsilon_r)$.

Property~\ref{F:3} shows that we can reconstruct the whole morphism $(p,w,\phi)$,
once we know $p$, $w$ and $\phi(\gamma)$ for some gallery $\gamma\in\Gamma_s$.
This morphism always exists if $(\gamma,\phi(\gamma))$ is chosen to be a $(p,w)$-pair.
The exact formulation is as follows.

\begin{lemma}\label{lemma:0}
Let $\gamma\in\Gamma_s$ and $\delta\in\Gamma_{s'}$ be galleries, $p:[1,|s|]\to[1,|s'|]$
be a monotone embedding and $w\in W$ such that $(\gamma,\delta)$ is a $(p,w)$-pair.
There exists a unique map\linebreak $\phi:\Gamma_s\to\Gamma_{s'}$ such that $\phi(\gamma)=\delta$
and $(p,w,\phi)$ is a morphism of $\Fold$.
\end{lemma}
\begin{proof} We want to extend $\phi$ to $\Gamma_s$ by the rule
$$
\phi(\f_{i_n}\cdots\f_{i_2}\f_{i_1}\gamma)=\f_{p(i_n)}\cdots\f_{p(i_2)}\f_{p(i_1)}\delta
$$
for any indices $i_1,\ldots,i_n\in[1,|s|]$. Obviously, we need to prove that $\phi$
given by this definition is well-defined. Once this is proved, the triple $(p,w,\phi)$ will be a morphism of $\Fold$,
as by Lemma~\ref{lemma:x2} and the definition of
a $(p,w)$-pair, we have
\begin{equation}\label{eq:3}
\bbeta_{p(i)}(\f_{p(i_n)}\cdots\f_{p(i_2)}\f_{p(i_1)}\delta)=\pm w\bbeta_i(\f_{i_n}\cdots\f_{i_2}\f_{i_1}\gamma).
\end{equation}
As the folding operators are idempotent, it suffices to prove that
\begin{equation}~\label{eq:1}
\f_{i_n}\cdots\f_{i_2}\f_{i_1}\gamma=\gamma,
\end{equation}
implies
\begin{equation}\label{eq:2}
\f_{p(i_n)}\cdots\f_{p(i_2)}\f_{p(i_1)}\delta=\delta.
\end{equation}
Applying $()^i$ to both sides of~(\ref{eq:1}), using~(\ref{eq:-2}) and cancelling $\gamma^i$, we get that~(\ref{eq:1}) is equivalent to
\begin{equation}\label{eq:4}
s_{\tau_n}^{\cond{i_n\le i}}\cdots s_{\tau_2}^{\cond{i_2\le i}}s_{\tau_1}^{\cond{i_1\le i}}=e\quad \forall i=1,\ldots,|s|,
\end{equation}
where $\tau_k=\bbeta_{i_k}(\f_{i_{k-1}}\cdots\f_{i_2}f_{i_1}\gamma)$.
Similarly,~(\ref{eq:2}) is equivalent to
\begin{equation}\label{eq:5}
s_{\sigma_n}^{\cond{p(i_n)\le j}}\cdots s_{\sigma_2}^{\cond{p(i_2)\le j}}s_{\sigma_1}^{\cond{p(i_1)\le j}}=e\quad \forall j=1,\ldots,|s'|,
\end{equation}
where $\sigma_k=\bbeta_{p(i_k)}(\f_{p(i_{k-1})}\cdots\f_{p(i_2)}\f_{p(i_1)}\delta)=\pm w\bbeta_{i_k}(\f_{i_{k-1}}\cdots\f_{i_2}\f_{i_1}\gamma)=\pm w\tau_k$.
We have applied ~(\ref{eq:3}) to obtain the middle equality.

Thus we must prove that~(\ref{eq:4}) implies~(\ref{eq:5}). Take an arbitrary $j=1,\ldots,|s'|$.
If $j$ is less than all elements $p(i_1),\ldots,p(i_n)$, then~(\ref{eq:5}) is simply the equality
$e=e$. Otherwise, take the maximal $k$
such that $p(i_k)\le j$. We have the following equivalences:
$$
p(i_m)\le j\Leftrightarrow p(i_m)\le p(i_k)\Leftrightarrow i_m\le i_k
$$
for all $m=1,\ldots,n$. Therefore, conjugating~(\ref{eq:4}) for $i=i_k$ with $w$, we get~(\ref{eq:5}).
\end{proof}

The proof of the following simple lemma is left to the reader.

\begin{lemma}\label{lemma:3}
Let $(p,w,\phi):s\to s'$ be a morphism of $\Fold$. Then $\phi$ is an embedding and $(p,w,\phi)$
is a monomorphism of $\Fold$.
For any $\gamma\in\Gamma_s$, a gallery $\delta\in\Gamma_{s'}$ belongs to $\im\phi$
if and only if $\delta_j=\phi(\gamma)_j$ for $j\in[1,|s'|]\setminus\im p$.
\end{lemma}

\subsection{Restriction of cohomologies along a morphism}
We are going to prove that morphisms of $\Fold$
give rise to restrictions on the level of cohomologies similarly to how\linebreak
$T$-equivariant continuous maps between $T$-spaces $Y\to X$ give rise to
maps $H^\bullet_T(X)\to H^\bullet_T(Y)$.

\begin{theorem}\label{theorem:1:rocaam}
Let $(p,w,\phi):s\to s'$ be a morphism of $\Fold$ and $g\in\X(s')$.
We define a map from $\Gamma_s$ to $S$ by $g_{(p,w,\phi)}(\gamma)=w^{-1}(g\circ\phi(\gamma))$. 
Then $g_{(p,\phi,w)}\in\X(s)$.
\end{theorem}

Before we prove this theorem, we explain some auxiliary constructions. For $\alpha\in R$, a sequence of simple roots $s$
and $X\subset[1,|s|]$, we define the following relation on $\Gamma_s$:
$$
\gamma\sim_\alpha^X\delta\Leftrightarrow\big(\gamma_i\ne\delta_i\Rightarrow \bbeta_i(\gamma)=\pm\alpha\text{ and }i\in X\big).
$$
This relation, being the intersection
of two equivalence relations $\sim_\alpha$ and `coincide outside $X$',
is itself an equivalence relation.
Clearly, $\f_i\gamma\sim_\alpha^X\gamma$ for any $i\in X\cap M_\alpha(\gamma)$.

\begin{lemma}\label{proposition:1:new} The relation
$\gamma\sim_\alpha^X\delta$ holds if and only if $\delta=\f_{i_n}\cdots\f_{i_1}\gamma$ for some $i_1,\ldots,i_n\in X$ such that
$i_k\in M_\alpha(\f_{i_{k-1}}\cdots\f_{i_1}\gamma)$.
\end{lemma}
\begin{proof} The above remark proves the `if'-part. Thus we only need to prove the `only if'-part.
Let $\gamma\sim_\alpha^X\delta$. We denote by $i$
the greatest number such that $\gamma_i\ne\delta_i$ or $-\infty$ if $\gamma=\delta$.
We apply induction on $i$, the case $i=-\infty$ being obvious.

Suppose now that $i\ge1$. By definition, we have $i\in X\cap M_\alpha(\gamma)$. Moreover,
$(\f_i\gamma)_j=\delta_j$ for $j\ge i$. By the inductive hypothesis, $\delta=\f_{i_n}\cdots\f_{i_1}\f_i\gamma$
for some $i_1,\ldots,i_n\in X$ such that $i_k\in M_\alpha(\f_{i_{k-1}}\cdots\f_{i_1}\f_i\gamma)$.
\end{proof}

\begin{lemma}\label{lemma:4}
Let $(p,w,\phi):s\to s'$ be a morphism of $\Fold$, $X\subset[1,|s|]$ and $\alpha\in R^+$.
For any $\gamma,\delta\in\Gamma_s$, $\gamma\sim_\alpha^X\delta$ is equivalent to
$\phi(\gamma)\sim_{w\alpha}^{p(X)}\phi(\delta)$.
\end{lemma}
\begin{proof} Let $\gamma\sim_\alpha^X\delta$. By Proposition~\ref{proposition:1:new}, we get
$\delta=\f_{i_n}\cdots\f_{i_1}\gamma$ for some $i_1,\ldots,i_n\in X$ such that
$i_k\in M_\alpha(\f_{i_{k-1}}\cdots\f_{i_1}\gamma)$.
Applying $\phi$, we get
$$
\phi(\delta)=\f_{p(i_n)}\cdots\f_{p(i_1)}\phi(\gamma).
$$
Since
$$
\bbeta_{p(i_k)}(\f_{p(i_{k-1})}\cdots\f_{p(i_1)}\phi(\gamma))=
\bbeta_{p(i_k)}(\phi(\f_{i_{k-1}}\cdots\f_{i_1}\gamma))
=\pm w\bbeta_{i_k}(\f_{i_{k-1}}\cdots\f_{i_1}\gamma)=\pm w\alpha,
$$
we get $\phi(\delta)\sim_{w\alpha}^{p(X)}\phi(\gamma)$ by Proposition~\ref{proposition:1:new}.

Conversely, let $\phi(\gamma)\sim_{w\alpha}^{p(X)}\phi(\delta)$. We denote by $i$ the maximal number such that
$\phi(\gamma)_i\ne\phi(\delta)_i$ or $-\infty$ if $\phi(\gamma)=\phi(\delta)$. We apply induction on $i$.
If $i=-\infty$, then $\phi(\gamma)=\phi(\delta)$, whence $\gamma=\delta$ by Lemma~\ref{lemma:3}.

Now suppose that $i\ge1$. As $i\in p(X)$, we have $i=p(j)$ for some $j\in X$.
By definition, we get
$$
\pm w\alpha=\bbeta_i(\phi(\gamma))=\bbeta_{p(j)}(\phi(\gamma))=\pm w\bbeta_j(\gamma),
$$
whence $\bbeta_j(\gamma)=\pm\alpha$. Moreover,
$$
\phi(\f_j\gamma)=\f_i\phi(\gamma)\sim_{w\alpha}^{p(X)}\phi(\gamma)\sim_{w\alpha}^{p(X)}\phi(\delta)
$$
and $(\f_i\phi(\gamma))_k=\phi(\delta)_k$ for any $k\ge i$. By the inductive hypothesis, we get $\f_j\gamma\sim_\alpha^X\delta$.
It remains to notice that $\gamma\sim_\alpha^X\f_j\gamma$.
\end{proof}

\begin{lemma}\label{lemma:5}
Let $(p,w,\phi):s\to s'$ be a morphism of $\Fold$ and $\gamma\in\Gamma_s$.
For any $\tau\in\Gamma_{s'}$, the following conditions are equivalent:
\begin{enumerate}
\item\label{lemma:5:p:1} $\tau\in\im\phi$, $\phi^{-1}(\tau)\sim_\alpha\gamma$;
\item\label{lemma:5:p:2} $\tau\sim_{w\alpha}^{\im p}\phi(\gamma)$.
\end{enumerate}
\end{lemma}
\begin{proof} By Lemma~\ref{lemma:4}, we get that~\ref{lemma:5:p:1} implies~\ref{lemma:5:p:2}.
Now suppose that~\ref{lemma:5:p:2} holds. In view of Lemma~\ref{lemma:4}, it suffices to note
that $\tau\in\im p$ by Lemma~\ref{lemma:3}.
\end{proof}

For $\alpha\in R$, $\gamma\in\Gamma_s$, $X\subset[1,|s|]$ and $\epsilon:X\to\{-1,1\}$, we set
$$
J_\alpha^{X,\epsilon}(\gamma)=\{i\in X\suchthat\bbeta_i(\gamma)=\epsilon_i\alpha\}.
$$
Thus $J_\alpha(\gamma)=J_\alpha^{[1,|s|],1}(\gamma)$. 

\begin{lemma}\label{lemma:6}
Let $(p,w,\phi):s\to s'$ be a morphism of $\Fold$ having sign $\epsilon$.
Then $p(J_\alpha(\gamma))=J_{w\alpha}^{\im p,\epsilon\circ p^{-1}}(\phi(\gamma))$
for any $\alpha\in R$.
\end{lemma}
\begin{proof}
The result follows from the following chain of equivalences:
\begin{multline*}
i\in p(J_\alpha(\gamma))\Leftrightarrow i\in\im p\text{ \& }p^{-1}(i)\in J_\alpha(\gamma)\Leftrightarrow
i\in\im p\text{ \& }\bbeta_{p^{-1}(i)}(\gamma)=\alpha\\
\Leftrightarrow i\in\im p\text{ \& }\epsilon_{p^{-1}(i)}w\bbeta_{p^{-1}(i)}(\gamma)=\epsilon_{p^{-1}(i)}w\alpha
\Leftrightarrow i\in\im p\text{ \& }\bbeta_i(\phi(\gamma))=\epsilon_{p^{-1}(i)}w\alpha\\
\Leftrightarrow i\in J_{w\alpha}^{\im p,\epsilon\circ p^{-1}}(\phi(\gamma)).
\end{multline*}
\end{proof}

\begin{proof}[Proof of Theorem~\ref{theorem:1:rocaam}] First we consider the case $k=\Z'$ in order to
apply Proposition~\ref{proposition:x:1}. Thus it suffices to prove that
$$
w^{-1}\sum_{\delta\in\Gamma_s,\delta\sim_\alpha\gamma,J_\alpha(\delta)\subset J_\alpha(\gamma)}(-1)^{|J_\alpha(\delta)|}g\circ\phi(\delta)\=0\pmod{\alpha^{|J_\alpha(\gamma)|}}
$$
for any $\alpha\in R^+$ and $\gamma\in\Gamma_s$.
By Lemmas~\ref{lemma:4},~\ref{lemma:5}, and~\ref{lemma:6} this equivalence
can be rewritten as follows:
\begin{equation}\label{eq:10}
\sum_{\ov{\tau\in\Gamma_{s'},\tau\sim_{w\alpha}^{\im p}\rho}{J_{w\alpha}^{\im p,\epsilon\circ p^{-1}}(\tau)\subset J_{w\alpha}^{\im p,\epsilon\circ p^{-1}}(\rho)}}(-1)^{|J_{w\alpha}^{\im p,\epsilon\circ p^{-1}}(\tau)|}g(\tau)\=0\pmod{(w\alpha)^{|J_{w\alpha}^{\im p,\epsilon\circ p^{-1}}(\rho)|}},
\end{equation}
where $\rho=\phi(\gamma)$ (and $\tau=\phi(\delta)$). 

To prove this formula, we recall some constructions from~\cite{Haerterich}.
Let $r'=|s'|$ and $\alpha'$ denote the positive of two roots
$w\alpha$ and $-w\alpha$. We consider the set $M_{\alpha'}(\rho)=\{m_1<\cdots<m_\l\}$.
H\"arterich~\cite[Section 4]{Haerterich} constructs the embedding
$v_\rho^{\alpha'}:(G_{\alpha'}/B_{\alpha'})^\l\hookrightarrow\SSigma_{s'}$ by requiring that its composition 
with the map $\iota:\SSigma_{s'}\hookrightarrow(G/B)^{r'}$ defined by~(\ref{eq:a})
be equal to\footnote{Here and in the rest of the proof, we write $g$ instead of $gB$, when speaking about elements of $G/B$.}
$$
\begin{tikzcd}
(g_1,\ldots,g_\l)\arrow[mapsto]{r}{\iota\circ v^{\alpha'}_\rho}&(\rho_{\min}^1,\ldots \rho_{\min}^{m_1-1},g_1\rho_{\min}^{m_1},\ldots g_1\rho_{\min}^{m_2-1},
\ldots,g_\l\rho_{\min}^{m_\l},\ldots g_\l\rho_{\min}^{r'}),
\end{tikzcd}
$$
where $\rho_{\min}$ is the minimal element in the $\sim_\alpha$-equivalence class of $\rho$
with respect to the total order $\vartriangleleft$.
Note that $(\im v^{\alpha'}_\rho)^T=\{\tau\in\Gamma_{s'}\suchthat\tau\sim_{\alpha'}\rho\}$.
The commutative diagram (see the notation of Section~\ref{rank2})
$$
\begin{tikzcd}
\BS_{\alpha',\l}^\2\arrow{r}{\iota_2}[swap]{\sim}&(G_{\alpha'}/B_{\alpha'})^\l\arrow[hook]{r}{v^{\alpha'}_\rho}&\SSigma_{s'}\\
\Gamma^\2_{\alpha',\l}\arrow{r}{\iota_2}[swap]{\sim}\arrow[hook]{u}&\{e,s_{\alpha'}\}^\l\arrow[hook]{r}{v^{\alpha'}_\rho}\arrow[hook]{u}&\Gamma_{s'}\arrow[hook]{u}
\end{tikzcd}
$$
yields the following commutative diagram of cohomologies:
\begin{equation}\label{eq:x}
\begin{tikzcd}
H_T^\bullet(\BS_{\alpha',\l}^\2,\Z')\arrow{d}&\arrow{l}[swap]{\iota_2^*}{\sim}H_T^\bullet((G_{\alpha'}/B_{\alpha'})^\l,\Z')\arrow{d}&\arrow{l}[swap]{(v^{\alpha'}_\rho)^*}H_T^\bullet(\SSigma_{s'},\Z')\arrow{d}\\
H_T^\bullet(\Gamma^\2_{\alpha',\l},\Z')&\arrow{l}[swap]{\iota_2^*}{\sim}H_T^\bullet(\{e,s_{\alpha'}\}^\l,\Z')&H_T^\bullet(\Gamma_{s'},\Z')\arrow{l}[swap]{(v^{\alpha'}_\rho)^*}
\end{tikzcd}
\end{equation}
Hence we get that $g\circ v^{\alpha'}_\rho\circ\iota_2$ belongs to the leftmost
vertical arrow.
We want to use the substitutions $v_\rho^{\alpha'}\circ\iota_2(\lm)=\rho$ and
$\tau=v_\rho^{\alpha'}\circ\iota_2(\mu)$, where $\lm,\mu\in\Gamma_{\alpha',\l}^\2$ in~(\ref{eq:10}).
Our next aim is to write the summation in this formula in the variables $\lm$ and $\mu$.

Let $s'_i=s_{\alpha'_i}$, where $\alpha'_i$ is a simple root, and $\epsilon$ be the sign of
$(p,w,\phi)$. We set
$$
\begin{array}{l}
A=\{j=1,\ldots,\l\suchthat m_j\in\im p\},\\[10pt]
\xi:A\to\{-1,1\},\quad \xi_j=\dfrac{\epsilon_{p^{-1}(m_j)}w\alpha}{\rho_{\min}^{m_j}(\alpha'_{m_j})},\\[16pt]
B(\nu)=\{j\in A\;\suchthat\;\nu^j(-\alpha')=\xi_j\alpha'\}=\{j\in A\;\suchthat\;\nu^j=s_{\alpha'}^{(\xi_j+1)/2}\}
\end{array}
$$
for $\nu\in\Gamma_{\alpha',\l}^\2$ in the last formula.

First, we calculate $\iota\circ v_\rho^{\alpha'}\circ\iota_2(\lm)$ as follows:
\begin{equation}\label{eq:-xx}
\!\!\!
\begin{array}{l}
\begin{tikzcd}
(\lm_1,\ldots,\lm_\l)\arrow[mapsto]{r}{\iota_2}&(\lm^1,\ldots,\lm^\l)&
\end{tikzcd}
\\
\begin{tikzcd}
\arrow[mapsto]{r}{\iota\circ v_\rho^{\alpha'}}&(\rho_{\min}^1,\ldots,\rho_{\min}^{m_1-1},\lm^1\rho_{\min}^{m_1},\ldots\lm^1\rho_{\min}^{m_2-1},\ldots,\lm^\l\rho_{\min}^{m_\l},\ldots,\lm^\l\rho_{\min}^{r'})
=(\rho^1,\rho^2,\ldots,\rho^{r'}).
\end{tikzcd}\!\!\!
\end{array}\!\!\!\!\!\!\!\!
\end{equation}
This formula and the similar formula for $\iota\circ v_\rho^{\alpha'}\circ\iota_2(\mu)$ imply
$$
\rho_i=
\left\{
\begin{array}{ll}
(\rho_{\min})_i&\text{ if }i\notin\{m_1,\ldots,m_\l\};\\
(\rho_{\min})_i&\text{ if }i=m_j\text{ and }\lm_j=e;\\
s_{\alpha'_i}(\rho_{\min})_i&\text{ if }i=m_j\text{ and }\lm_j=s_{\alpha'},
\end{array}
\right.\quad
\tau_i=
\left\{
\begin{array}{ll}
(\rho_{\min})_i&\text{ if }i\notin\{m_1,\ldots,m_\l\};\\
(\rho_{\min})_i&\text{ if }i=m_j\text{ and }\mu_j=e;\\
s_{\alpha'_i}(\rho_{\min})_i&\text{ if }i=m_j\text{ and }\mu_j=s_{\alpha'}.
\end{array}
\right.
$$
Hence we get that the restriction $\tau\sim_{w\alpha}^{\im p}\rho$ from the summation of~(\ref{eq:10})
is equivalent to the following one:
$$
\lm_j\ne\mu_j\Longrightarrow j\in A.
$$

Now let us calculate $J_{w\alpha}^{\im p,\epsilon\circ p^{-1}}(\rho)$.
%
For any $j\in A$, we have $\rho^{m_j}=\lm^j\rho_{\min}^{m_j}$ by~(\ref{eq:-xx}), whence
\begin{multline*}
\bbeta_{m_j}(\rho)=\epsilon_{p^{-1}(m_j)}w\alpha
\Leftrightarrow
\lm^j\rho_{\min}^{m_j}(-\alpha'_{m_j})=\epsilon_{p^{-1}(m_j)}w\alpha\Leftrightarrow\\
\Leftrightarrow\lm^j\left(-\alpha'\frac{\rho_{\min}^{m_j}(\alpha'_{m_j})}{\alpha'}\right)=\epsilon_{p^{-1}(m_j)}w\alpha
\Leftrightarrow\lm^j\left(-\alpha'\right)=\frac{\epsilon_{p^{-1}(m_j)}w\alpha}{\rho_{\min}^{m_j}(\alpha'_{m_j})}\alpha'
=\xi_j\alpha'.
\end{multline*}
In this calculation, we used ${\rho_{\min}^{m_j}(\alpha'_{m_j})}/{\alpha'}=\pm1$.
From the above chain of equivalences, we get
\begin{multline*}
J_{w\alpha}^{\im p,\epsilon\circ p^{-1}}(\rho)=
\{i\in \im p\suchthat\bbeta_i(\rho)=\epsilon_{p^{-1}(i)}w\alpha\}\\[6pt]
=\{i\in \im p\cap M_{\alpha'}(\rho)\suchthat\bbeta_i(\rho)=\epsilon_{p^{-1}(i)}w\alpha\}
=m_{\{j\in A\suchthat\bbeta_{m_j}(\rho)=\epsilon_{p^{-1}(m_j)}w\alpha \}}
=m_{B(\lm)}.
\end{multline*}
Similarly, $J_{w\alpha}^{\im p,\epsilon\circ p^{-1}}(\tau)=m_{B(\mu)}$.
Therefore the inclusion $J_{w\alpha}^{\im p,\epsilon\circ p^{-1}}(\tau)\subset J_{w\alpha}^{\im p,\epsilon\circ p^{-1}}(\rho)$
is equivalent to $B(\mu)\subset B(\lm)$.
Thus we can rewrite~(\ref{eq:10}) as follows:
\begin{equation}\label{eq:xx}
\sum_{\ov{\ov{\mu\in\Gamma_{\alpha',\l}^\2}{\mu_j\ne\lm_j\Rightarrow j\in A}}{ B(\mu)\subset B(\lm)}}\frac{f(\mu)}{(-1)^{|B(\mu)|}\alpha'^{\,|B(\lm)|}}\in S_{\Z'}.
\end{equation}
for any $f$ belonging to the image of the restriction $H^\bullet_T(\BS_{\alpha',\l}^\2,\Z')\to H^\bullet_T(\Gamma_{\alpha',\l}^\2,\Z')$.
Actually $f=g\circ v^{\alpha'}_\rho\circ\iota_2$, see the explanation immediately after Diagram~(\ref{eq:x}).

This fact follows from Lemma~\ref{lemma:y} for $\alpha=\alpha'$ if we set
$$
L=\{1,\ldots,\l\}\setminus A,\quad  K=\{j\in A\suchthat \lm^j=s_{\alpha'}^{(1-\xi_j)/2}\}
$$
and define the functions $l:L\to\{e,s_{\alpha'}\}$ and $k:K\to\{e,s_{\alpha'}\}$ by: $l_j=\lm_j$ and $k_j=\lm^j$.

We note that the inclusion $B(\mu)\subset B(\lm)$ is equivalent to
$$
\mu^j=s_{\alpha'}^{(\xi_j+1)/2}\Rightarrow \lm^j=s_{\alpha'}^{(\xi_j+1)/2},
$$
which in its turn is equivalent to
$$
\lm^j=s_{\alpha'}^{(1-\xi_j)/2}\Rightarrow \mu^j=\lm^j. 
$$
The last condition can be written as $\forall j\in K: \mu^j=k_j$.

It remains to calculate the denominator in the sum of Lemma~\ref{lemma:y} and compare
it with the denominator of~(\ref{eq:xx}). We get
\begin{multline*}
\prod_{j\in[1,\l]\setminus(K\cup L)}\mu^j(-\alpha')=\prod_{j\in B(\lm)}\mu^j(-\alpha)
=\prod_{j\in B(\lm)\setminus B(\mu)}\mu^j(-\alpha')\cdot \prod_{j\in B(\mu)}\mu^j(-\alpha')\\
=\prod_{j\in B(\lm)\setminus B(\mu)}s_{\alpha'}^{(1-\xi_j)/2}(-\alpha')\cdot\prod_{j\in B(\mu)}s_{\alpha'}^{(\xi_j+1)/2}(-\alpha')\\
=(-1)^{|B(\mu)|}\prod_{j\in B(\lm)}s_{\alpha'}^{(1-\xi_j)/2}(-\alpha')
=\omega(-1)^{|B(\mu)|}{\alpha'}^{\,|B(\lm)|},
\end{multline*}
where $\omega=\pm1$ does not depend on $\mu$. This proves~(\ref{eq:xx}).

We can now come back to the general case. We apply Diagram~(\ref{eq:-3}).
Any element $g\in\X(s')$ has the following form $g=\zeta\Big(\sum_{i=1}^nf_i\otimes\lm_i\Big)$,
where $f_i\in\X(s',\Z')$ and $\lm_i\in S$. As the homomorphism $\zeta':S_{\Z'}\to S$ is $W$-invariant,
we get
$$
g_{(p,w,\phi)}=\zeta\(\sum_{i=1}^n(f_i)_{(p,w,\phi)}\otimes w^{-1}\lm_i\).
$$
Indeed, for any $\gamma\in\Gamma_s$, we have
\begin{multline*}
\displaystyle
g_{(p,w,\phi)}(\gamma)=w^{-1}g(\phi(\gamma))=w^{-1}\zeta\Big(\sum_{i=1}^nf_i\otimes\lm_i\Big)(\phi(\gamma))
=\sum_{i=1}^nw^{-1}\zeta'\big(f_i(\phi(\gamma))\big)\cdot w^{-1}\lm_i\\
\displaystyle
\!=\sum_{i=1}^n\zeta'\big(w^{-1}f_i(\phi(\gamma))\big)\cdot w^{-1}\lm_i
=\!\sum_{i=1}^n\zeta'((f_i)_{(p,w,\phi)}(\gamma))\cdot w^{-1}\lm_i
=\zeta\(\sum_{i=1}^n(f_i)_{(p,w,\phi)}\otimes w^{-1}\lm_i\)\!(\gamma).\!\!\!\!\!
\end{multline*}
It follows from the first part of the proof that $(f_i)_{(p,w,\phi)}\in\X(s,\Z')$. Hence $g\in\X(s)$.
\end{proof}

This theorem has the following corollary. Consider the field $Q$ of quotients of $S$ and the natural
product $\Map(\Gamma_s,Q)\times\Map(\Gamma_s,Q)\to Q$ defined by
$(f,g)=\sum_{\gamma\in\Gamma_s}f(\gamma)g(\gamma)$. We define the dual $S$-module
$$
D\X(s)=\{f\in\Map(\Gamma_s,Q)\suchthat(f,g)\in S\;\forall g\in\X(s)\}.
$$

\begin{corollary}
Let $(p,w,\phi):s\to s'$ be a morphism of $\Fold$ and $f\in D\X(s)$.
We define the map from $\Gamma_{s'}$ to $Q$ by
$$
f^{(p,w,\phi)}(\delta)=
\left\{
\begin{array}{ll}
wf(\gamma)&\text{ if }\delta=\phi(\gamma);\\
0&\text{ if }\delta\notin\im\phi.
\end{array}
\right.
$$
Then $f^{(p,w,\phi)}\in D\X(s')$.
\end{corollary}
\begin{proof} For any $g\in\X(s')$, we have
$$
(f^{(p,w,\phi)},g)=\sum_{\gamma\in\Gamma_s}f^{(p,w,\phi)}(\phi(\gamma))\cdot g(\phi(\gamma))
=\sum_{\gamma\in\Gamma_s}wf(\gamma)\cdot g(\phi(\gamma))=w(f,g_{(p,w,\phi)})\in S.
$$
\end{proof}


\subsection{\texorpdfstring{Functor $\widetilde H$}{Functor tilde H}}\label{widetildeH}

First we describe the embedding $\Seq\hookrightarrow\Fold$.
It acts identically on objects and identifies a morphism $p:s\to s'$ with the morphism $(p,e,\phi^p)$,
where
$$
\phi^p(\gamma)_j=
\left\{
\!\!
\begin{array}{ll}
\gamma_i&\text{ if }j=p(i);\\[3pt]
e&\text{ if }j\notin\im p.
\end{array}
\right.
$$
Note that this morphism is of identical rotation and positive sign.

Now we construct the functor $\tH$. 
From the commutativity of Diagram~(\ref{eq:cdf}), we obtain $\tH(s)=H^\bullet_T(\SSigma_s)$.
It remains to define $\tH$ on morphisms.

Let $(p,w,\phi):s\to s'$ be a morphism of $\Fold$. We define $\tH((p,w,\phi))$ uniquely by claiming
that the following diagram be commutative:
\begin{equation}\label{eq:zz}
\begin{tikzcd}[column sep=6em]
H_T^\bullet(\SSigma_{s'})\arrow{r}{\tH((p,w,\phi))}\arrow[hook]{d}&H_T^\bullet(\SSigma_s)\arrow[hook]{d}\\
H_T^\bullet(\Gamma_{s'})\arrow{r}{?_{(p,w,\phi)}}&H_T^\bullet(\Gamma_s)
\end{tikzcd}
\end{equation}
Here the morphism $?_{(p,w,\phi)}$ of the bottom row comes from Theorem~\ref{theorem:1:rocaam},
which guarantees the existence of $\tH((p,w,\phi))$. The uniqueness follows from
the fact that the right vertical arrow is an embedding.

We need to prove that $\tH$ is a functor. It obviously takes identical morphisms to identical morphisms.
To prove that $\tH$ takes
compositions to compositions (in the reversed order), consider the diagram
$$
\begin{tikzcd}[column sep=6em]
H_T^\bullet(\SSigma_{s''})\arrow[swap]{r}{\tH((p',w',\phi'))}\arrow[hook]{d}\arrow[bend left=10pt]{rr}{\tH((p',w',\phi')\circ(p,w,\phi))}&H_T^\bullet(\SSigma_{s'})\arrow[hook]{d}\arrow{r}[swap]{\tH((p,w,\phi))}&H_T^\bullet(\SSigma_s)\arrow[hook]{d}\\
H_T^\bullet(\Gamma_{s''})\arrow{r}{?_{(p',w',\phi')}}\arrow[bend right=10pt]{rr}[swap]{?_{(p'p,w'w,\phi'\phi)}}&H_T^\bullet(\Gamma_{s'})\arrow{r}{?_{(p,w,\phi)}}&H_T^\bullet(\Gamma_s)
\end{tikzcd}
$$
for two morphisms $(p,w,\phi):s\to s'$ and $(p',w',\phi'):s'\to s''$ of $\Fold$.
Here all rectangles and triangles except the top one are commutative. As the vertical arrows are injective the top triangle is also
commutative.

To prove that Diagram~(\ref{eq:cdf}) is commutative, we need to prove that
for any morphism $p:s\to s'$ of $\Seq$, the diagram
$$
\begin{tikzcd}[column sep=6em]
H_T^\bullet(\SSigma_{s'})\arrow{r}{\BS(p)^*}\arrow[hook]{d}&H_T^\bullet(\SSigma_s)\arrow[hook]{d}\\
H_T^\bullet(\Gamma_{s'})\arrow{r}{?_{(p,e,\phi^p)}}&H_T^\bullet(\Gamma_s)
\end{tikzcd}
$$
is commutative. This is however true as $g_{(p,e,\phi^p)}=g\circ\BS(p)|_{\Gamma_s}$
for any $g\in H_T^\bullet(\Gamma_{s'})$, which in turn follows from $\phi^p=\BS(p)|_{\Gamma_s}$.

\section{Topology and Intermediate Categories}\label{Topology}

\subsection{Transition sequences} We discuss here how to determine if a point $[c]^\gamma_s\in U^\gamma_s$ 
belongs to a different chart $U^\delta_s$ and if so how to compute its new coordinates.
In the proof of the following proposition, we use the operations $\tr_i$ and $\pi_i$ defined in Section~\ref{BS}.

\begin{proposition}\label{lemma:1}
Let $s$ be a sequence of simple reflections, where $s_i=s_{\alpha_i}$ for some $\alpha_i\in\Pi$,
$\gamma,\delta\in\Gamma_s$ and $[c]_s^\gamma$ be a point of $U^\gamma_s$. We try to define
the elements $b_0,\ldots,b_{|s|}\in B$ and the numbers $d_1,\ldots,d_{|s|}\in\C$ inductively as follows:
\begin{enumerate}
\item\label{lemma:1:step:1} $b_0=e$;
\item\label{lemma:1:step:2} $b_ix_{\gamma_{i+1}(-\alpha_{i+1})}(c_{i+1})\gamma_{i+1}=x_{\delta_{i+1}(-\alpha_{i+1})}(d_{i+1})\delta_{i+1}b_{i+1}$;
\end{enumerate}
The following conditions are equivalent:
{\renewcommand{\labelenumi}{{\rm\theenumi}}
\renewcommand{\theenumi}{{\rm(\roman{enumi})}}
\begin{enumerate}
\item\label{lemma:1:cond:1} The above algorithm does not stop prematurely;
\item\label{lemma:1:cond:2} $[c]_s^\gamma\in U^\delta_s$.
\end{enumerate}}
If these conditions are satisfied, then $[c]_s^\gamma=[d]_s^\delta$ and
$b_i=(\llbracket d_1,\ldots,d_i\rrbracket^\delta)^{-1}\llbracket c_1,\ldots,c_i\rrbracket^\gamma$.
\end{proposition}
\begin{proof} We apply the induction on $|s|$. The implication~\ref{lemma:1:cond:1}$\Rightarrow$\ref{lemma:1:cond:2} is obvious. Let us suppose now
that~\ref{lemma:1:cond:2} holds. We write $[c]^\gamma_s=[a]^\delta_s$ for the corresponding coordinates.
Suppose that our algorithm has already produced $b_0,\ldots,b_i$ and $d_1,\ldots,d_i$,
where $i<|s|$. By the inductive hypothesis,
$\displaystyle
[c_1,\ldots,c_i]_{(s_1,\ldots,s_i)}^{(\gamma_1,\ldots,\gamma_i)}=[d_1,\ldots,d_i]_{(s_1,\ldots,s_i)}^{(\delta_1,\ldots,\delta_i)}
$
and
$$
b_i=
\Big(\llbracket d_1,\ldots,d_i\rrbracket_{(s_1,\ldots,s_i)}^{(\delta_1,\ldots,\delta_i)}\Big)^{-1}\llbracket c_1,\ldots,c_i\rrbracket_{(s_1,\ldots,s_i)}^{(\gamma_1,\ldots,\gamma_i)}
=(\llbracket d_1,\ldots,d_i\rrbracket_s^\delta)^{-1}\llbracket c_1,\ldots,c_i\rrbracket_s^\gamma.
$$
Applying $\tr_i$ to $[c]^\gamma_s=[a]^\delta_s$, we get
$$
[d_1,\ldots,d_i]_{(s_1,\ldots,s_i)}^{(\delta_1,\ldots,\delta_i)}=[c_1,\ldots,c_i]^{(\gamma_1,\ldots,\gamma_i)}_{(s_1,\ldots,s_i)}=[a_1,\ldots,a_i]^{(\delta_1,\ldots,\delta_i)}_{(s_1,\ldots,s_i)}
$$
Hence $a_j=d_j$ for $j\le i$. Therefore, applying $\pi_{i+1}$ to $[c]^\gamma_s=[a]^\delta_s$, we get
$$
\llbracket c_1,\ldots,c_i\rrbracket_s^\gamma x_{\gamma_{i+1}(-\alpha_{i+1})}(c_{i+1})\gamma_{i+1}B=
\llbracket d_1,\ldots,d_i\rrbracket_s^\delta x_{\delta_{i+1}(-\alpha_{i+1})}(a_{i+1})\delta_{i+1}B,
$$
which allows us to compute $d_{i+1}=a_{i+1}$ and
$$
b_{i+1}=(x_{\delta_{i+1}(-\alpha_{i+1})}(a_{i+1})\delta_{i+1})^{-1}b_i x_{\gamma_{i+1}(-\alpha_{i+1})}(c_{i+1})\gamma_{i+1}=
(\llbracket d_1,\ldots,d_{i+1}\rrbracket^\delta)^{-1}\llbracket c_1,\ldots,c_{i+1}\rrbracket^\gamma.
$$
\end{proof}

\begin{remark}\rm
It follows from the Bruhat decomposition that $d_{i+1}$ and $b_{i+1}$ in step~\ref{lemma:1:step:2} are defined uniquely
(if they exist).
\end{remark}

The sequence $(b_1,\ldots,b_{|s|})$ in this proposition is called the {\it transition sequence}\label{ts}
of the point $[c]_s^\gamma$ from $U^\gamma_s$ to $U^\delta_s$. Computation of these sequences is
our main technical tool, when dealing with different charts of Bott-Samelson varieties.

The proof of the following simple corollary is left to  the reader.

\begin{corollary}\label{lemma:ww} Let $s$ be a sequence of simple roots, $\gamma\in\Gamma_s$ and $i=1,\ldots,|s|$.
If a point $[c]_s^\gamma$ belongs to $U_s^{\f_i\gamma}$, then $c_i\ne0$ and
the $k$th coordinate of this point in $U_s^{\f_i\gamma}$ is $c_k$ for $k<i$ and is $c_i^{-1}$ for $k=i$.
\end{corollary}

The following lemma performs step~\ref{lemma:1:step:2} of Proposition~\ref{lemma:1} in an important special case.

\begin{lemma}\label{lemma:x:1}
Suppose that the root system of $G$ is simply laced. Let $\alpha$ be a simple root, $\gamma\in\{e,s_\alpha\}$, $c\in\C$
and
$$
b=t_0x_{\beta_1}(a_1)t_1\cdots t_{m-1}x_{\beta_m}(a_m)t_m,
$$
where $t_0,\ldots,t_m\in T$ and $\beta_i>0$.
Suppose that there exist $c'\in\C$ and $b'\in B$ such that
$$
bx_{\gamma(-\alpha)}(c)\gamma=x_{\gamma(-\alpha)}(c')\gamma b'.
$$
Then $b'$ is a product of some elements of $T$ and root elements of $B$ having the form $x_{\beta_i}(a_i')$ or
$x_{\gamma\beta_i}(a_i'')$.
\end{lemma}
\begin{proof} 
First consider the case $\gamma=s_\alpha$. We are going to successfully carry factors $x_{\beta_i}(a_i)$ of $b$ over
elements of the form $x_\alpha(\hat c)s_\alpha$. If $\beta_i=\alpha$, then $x_{\beta_i}(a_i)x_\alpha(\hat c)s_\alpha=x_\alpha(a_i+\hat c)s_\alpha$.
We take then the previous factor $x_{\beta_{i-1}}(a_{i-1})$ if $i>0$.

If $\beta_i\ne\alpha$, then by~(\ref{eq:comm}), we get
$$
x_{\beta_i}(a_i)x_\alpha(\hat c)s_\alpha=
x_\alpha(\hat c)x_{\beta_i}(a_i)x_{s_\alpha\beta_i}(\epsilon\hat c\beta_i)s_\alpha
=x_\alpha(\hat c)s_\alpha x_{s_\alpha\beta_i}(\pm a_i)x_{\beta_i}(\pm\epsilon\hat c\beta_i)
$$
for some $\epsilon\in\{-1,0,1\}$.
As $\beta_i\in R^+\setminus\{\alpha\}$, we get $s_\alpha\beta_i>0$.

Now consider the case $\gamma=e$. We need to successfully carry factors
$x_{\beta_i}(a_i)$ over elements of the form $x_{-\alpha}(\hat c)$. If $\beta_i\ne\alpha$, then by~(\ref{eq:comm}), we get
$$
x_{\beta_i}(a_i)x_{-\alpha}(\hat c)=x_{-\alpha}(\hat c)x_{\beta_i}(a_i)x_{s_\alpha\beta_i}(\epsilon a_i\hat c)
$$
for some $\epsilon\in\{-1,0,1\}$. As $\beta_i\ne\alpha$, we get $s_\alpha\beta_i>0$.

Now suppose that $\beta_i=\alpha$. If $a_i\hat c\ne-1$, then~(\ref{eq:comm:1}) implies
$$
x_{\beta_i}(a_i)x_{-\alpha}(\hat c)=x_{-\alpha}\(\tfrac{\hat c}{a_i\hat c+1}\)x_{\beta_i}(a_i(a_i\hat c+1))h_\alpha(a_i\hat c+1).
$$
If, on the contrary, $a_i\hat c=-1$, then by~(\ref{eq:comm:2}), we get
$$
x_{\beta_i}(a_i)x_{-\alpha}\(\hat c\)=s_\alpha x_{\beta_i}(\hat c)h_\alpha(-\hat c).
$$
We are now in a case similar to the case $\gamma=s_\alpha$ considered above. As a result, we get
$$
bx_{-\alpha}(c)=x_\alpha(\tilde c)s_\alpha\tilde b
$$
for some $\tilde c\in\C$ and $\tilde b$ having the same form as $b'$ in the formulation of the lemma.
If $\tilde c\ne0$, then (as follows from the calculations of the case $\gamma=s_\alpha$) we have
$\alpha=\beta_j$ for some $j$ and the factor $x_{\beta_j}(a'_j)$ is absent in the product $\tilde b$.
Hence
$$
bx_{-\alpha}(c)=x_\alpha(\tilde c)s_\alpha\tilde b=x_\alpha(\tilde c)s_\alpha(-\tilde c)t\tilde b
=x_{-\alpha}(\tilde c^{-1})x_{\beta_j}(-\tilde c)t\tilde b.
$$
for some $t\in T$. It remains to set $c'=\tilde c^{-1}$ and $b'=x_{\beta_j}(-\tilde c)t\tilde b$.

Finally, in the case $\tilde c=0$, we get
$$
s_\alpha x_\alpha(\pm c')s_\alpha^{-1}b'=x_{-\alpha}(c')b'=bx_{-\alpha}(c)=s_\alpha\tilde b
$$
Cancelling out $s_\alpha$, we get $x_\alpha(\pm c')s_\alpha^{-1}b'=\tilde b$, which contradicts the Bruhat decomposition.
\end{proof}

\begin{corollary}\label{corollary:1}
Suppose that the root system of $G$ is simply laced. Let $(b_1,\ldots,b_r)$ be the transition sequence of $x$
from $U^\gamma_s$ to $U^{\f_i\gamma}_s$. Then for $k=i,\ldots,r$, every
$b_k$ is a product of elements of $T$ and the root elements of the form $x_{\tau_k\cdots\tau_{i+1}\alpha_i}(a)$,
where $\tau_j\in\{e,s_j\}$ and $\tau_k\cdots\tau_{i+1}\alpha_i>0$. Moreover, $b_k=e$ for $k<i$.
\end{corollary}
\begin{proof}
Let $s_i=s_{\alpha_i}$ for $\alpha_i\in\Pi$ and $x=[c]^\gamma_s$. As $x\in U^{\f_i\gamma}_s$, we get $c_i\ne0$ by Corollary~\ref{lemma:ww}.
If $\gamma_i=e$, then
\begin{equation}\label{eq:xminusalpha}
x_{-\alpha_i}(c_i)=x_{\alpha_i}(c_i^{-1})s_{\alpha_i}(-c_i^{-1})x_{\alpha_i}(c_i^{-1})=x_{\alpha_i}(c_i^{-1})s_itx_{\alpha_i}(c_i^{-1})
\end{equation}
for some $t\in T$. If $\gamma_i=s_i$, then
\begin{equation}\label{eq:xalpha}
x_{\alpha_i}(c_i)s_i=x_{\alpha_i}(c_i)s_{\alpha_i}(-c_i)t'=x_{\alpha_i}(c_i)s_{\alpha_i}(-c_i)x_{\alpha_i}(c_i)x_{\alpha_i}(-c_i)t'
=x_{-\alpha}(c_i^{-1})x_{\alpha_i}(-c_i)t'
\end{equation}
for some $t'\in T$. It remains to apply Lemma~\ref{lemma:x:1}.
\end{proof}

\subsection{$T$-curves} Let $X$ be a complex algebraic variety with an algebraic action of a complex torus $T$.
A {\it $T$-curve} is by definition the closure of a one-dimensional orbit of $T$. We say that a $T$-curve {\it connects} $T$-fixed
points $x$ and $y$ if and only if it contains both these points.

\begin{proposition}\label{proposition:T-curve} Let $X$ and $Y $ be complex algebraic varieties with algebraic actions of $T$ and
$\psi:X\to Y$ be a $T$-equivariant map continuous in the metric topology. If two points $x,x'\in X^T$ are connected by a $T$-curve on $X$,
then their images $\psi(x)$ and $\psi(x')$ either coincide or are connected by a $T$-curve on $Y$.
\end{proposition}
\begin{proof} We suppose that $\psi(x)\ne\psi(x')$. Let $z$ be a point of $X$ such that $\dim Tz=1$ and $x,x'\in A=\overline{Tz}$.
Consider the morphism $\lm:T\to A$ defined by $\lm(t)=tz$. It is obviously dominant.
By~\cite[Theorem 4.1]{Hum}, every irreducible component of $\lm^{-1}(z)=\Stab_T (z)$ has dimension $\dim T-1$.
We consider the closure $B=\overline{T\psi(z)}$ and denote by $n$ its dimension.
By continuity, we get $\psi(x),\psi(x')\in\psi(\overline{Tz})\subset\overline{\psi(Tz)}=B$.
Thus $B$ is not a point and $n\ge1$.

Let $\mu:T\to B$ be the dominant morphism defined by $\mu(t)=t\psi(z)$.
We consider the Zariski open subset $U\subset T\psi(z)\subset B$ as in~\cite[Theorem~4.3]{Hum}.
By part (b) of this theorem, the dimension of every irreducible component
of $\mu^{-1}(t\psi(z))=\Stab_T(t\psi(z))$ is $\dim T-n$ as soon as $t\psi(z)\in U$.
As $\Stab_T(t\psi(z))$ and $\Stab_T(\psi(z))$ are conjugated the same is true about irreducible components
of the latter variety. We clearly have $\Stab_T(z)\subset\Stab_T\psi(z)$ and every irreducible component
of the former variety is contained in some irreducible component of the latter.
Hence $n=1$. 
\end{proof}

\begin{remark}\rm
In this proposition, the map $\psi$ is not supposed to be a morphism of algebraic varieties,
only continuous and $T$-equivariant. In its proof,
we used the coincidence of the metric and Zariski closures for constructible subsets of algebraic varieties.
\end{remark}

We are going to prove a criterion for the existence of $T$-curves connecting
certain points of Bott-Samelson varieties.

\begin{lemma}\label{lemma:fcat:6.5}
Let $\gamma\in\Gamma_s$ and $i=1,\ldots,|s|$. A $T$-curve on $\SSigma_s$ connecting $\gamma$ and $\f_i\gamma$ exists if and only if
there is no $j>i$ such that $\bbeta_i(\gamma)=-\bbeta_j(\gamma)$.
\end{lemma}
\begin{proof} We set $r=|s|$.
Suppose that some $T$-curve connects points $\gamma$ and $\f_i\gamma$ but
$\bbeta_i(\gamma)=-\bbeta_j(\gamma)$ for some $j\in(i,r]$.
We take $j$ minimal satisfying this property.
Let $z$ be an arbitrary not $T$-stable point on our curve. Then our $T$-curve is equal to $\overline{Tz}$.

As $Tz$ is dense in $\overline{Tz}$, we get $tz\in U^\gamma_s$ for some $t\in T$. But $U^\gamma_s$ is $T$-stable,
whence $z\in U^\gamma_s$. Similarly $z\in U^{\f_i\gamma}_s$. Let us write $z=[c]^\gamma_s=[d]^{\f_i\gamma}_s$
for the corresponding coordinates. Corollary~\ref{lemma:ww} 
shows that $c_k=d_k$ for $k<i$ and $d_i=c_i^{-1}$.
If $c_j$ were not equal to $0$, then the product of the $i$th and the $j$th coordinates of any point of the orbit $Tz$
would be equal to $c_ic_j\ne0$. Thus the point $\gamma=[0]^\gamma_s$ would not belong to the closure of $Tz$ in $U^\gamma_s$
and hence also to the closure $\overline{Tz}$ in the ambient space $\SSigma_s$. Therefore, we assume $c_j=0$ in what follows.

Moreover, as $\dim Tz=1$, the roots of either set $\{\bbeta_k(\gamma)\suchthat c_k\ne0\}$ and $\{\bbeta_k(\f_i\gamma)\suchthat d_k\ne0\}$
lie on the same line and are thus equal to either $\alpha$ or $-\alpha$, where $\alpha=\bbeta_i(\gamma)$.

By the minimality of $j$, we get $\bbeta_k(\gamma)=\alpha$ for any $k\in M_\alpha(\gamma)\cap[i,j)$. Let $(b_1,\ldots,b_r)$ be the transition sequence of $z$
from $U^\gamma_s$ to $U^{\f_i\gamma}_s$. For each integer $l\ge i$, we denote by $l'$ the maximal element
of $M_\alpha(\gamma)$ not greater than $l$.

We are going to prove inductively on $l=i,\ldots,j-1$ that
$
b_l=t_lx_{\gamma_l\cdots\gamma_{l'+1}\alpha_{l'}}(u_l),
$
where $u_l\ne0$ and $t_l\in T$.
As $c_k=d_k$ for $k<i$, we get $b_1=\cdots=b_{i-1}=e$. Thus the case $l=i$ follows
from~(\ref{eq:xminusalpha}) or~(\ref{eq:xalpha}).

We suppose now that the inductive claim holds for $l=i,\ldots,j-2$ and are going to prove it for $l+1$.
Fist consider the case $l+1\notin M_\alpha(\gamma)$. In this case, $c_{l+1}=0$ and we get
$$
b_l\gamma_{l+1}=t_lx_{\gamma_l\cdots\gamma_{l'+1}\alpha_{l'}}(u_l)\gamma_{l+1}
=t_l\gamma_{l+1}x_{\gamma_{l+1}\gamma_l\cdots\gamma_{l'+1}\alpha_{l'}}(\pm u_l).
$$
Thus it suffices to prove that $\gamma_{l+1}\gamma_l\cdots\gamma_{l'+1}\alpha_{l'}>0$.
Suppose the contrary holds. From $b_l\in B$ and $u_l\ne0$, we know $\gamma_l\cdots\gamma_{l'+1}\alpha_{l'}>0$.
Hence $\gamma_{l+1}=s_{l+1}$ and $\gamma_l\cdots\gamma_{l'+1}\alpha_{l'}=\alpha_{l+1}$.
It follows from this equality that $\bbeta_{l+1}(\gamma)=-\bbeta_{l'}(\gamma)=-\alpha$, which contradicts
the fact $l+1\notin M_\alpha(\gamma)$. Note that we have additionally proved that $d_{l+1}=0$.

Now consider the case $l+1\in M_\alpha(\gamma)$. 
It follows from the equality $\bbeta_{l'}(\gamma)=\alpha=\bbeta_{l+1}(\gamma)$ that $\gamma_l\cdots\gamma_{l'+1}\alpha_{l'}=\gamma_{l+1}(\alpha_{l+1})$.
Hence we can write $b_l=t_lx_{\gamma_{l+1}(\alpha_{l+1})}(u_l)$. If $u_lc_{l+1}\ne-1$, then by~(\ref{eq:comm:1}), 
we get
\begin{multline*}
x_{\gamma_{l+1}(\alpha_{l+1})}(u_l)x_{\gamma_{l+1}(-\alpha_{l+1})}(c_{l+1})\gamma_{l+1}\\
=x_{\gamma_{l+1}(-\alpha_{l+1})}\(\tfrac{c_{l+1}}{u_lc_{l+1}+1}\)x_{\gamma_{l+1}(\alpha_{l+1})}(u_l(u_lc_{l+1}+1))h_{\gamma_{l+1}(\alpha_{l+1})}(u_lc_{l+1}+1)\gamma_{l+1}\\
=x_{\gamma_{l+1}(-\alpha_{l+1})}\(\tfrac{c_{l+1}}{u_lc_{l+1}+1}\)\gamma_{l+1}x_{\alpha_{l+1}}(\pm u_l(u_lc_{l+1}+1))h_{\alpha_{l+1}}(u_lc_{l+1}+1).
\end{multline*}
This proves the inductive claim for $l+1$ as $(l+1)'=l+1$. Now consider the case $u_lc_{l+1}=-1$. From~(\ref{eq:comm:2}), 
we get
%
\begin{equation}\label{eq:p}
\begin{array}{l}
b_lx_{\gamma_{l+1}(-\alpha_{l+1})}(c_{l+1})\gamma_{l+1}=t_lx_{\gamma_{l+1}(\alpha_{l+1})}(u_l)x_{\gamma_{l+1}(-\alpha_{l+1})}(c_{l+1})\gamma_{l+1}\\[6pt]
\hspace{98pt}=t_ls_{\gamma_{l+1}(\alpha_{l+1})}x_{\gamma_{l+1}(\alpha_{l+1})}(c_{l+1})h_{\gamma_{l+1}(\alpha_{l+1})}(-c_{l+1})\gamma_{l+1}\\[6pt]
\hspace{208pt}=t_l\gamma_{l+1}s_{\alpha_{l+1}}(\pm1)x_{\alpha_{l+1}}(\pm c_{l+1})h_{\alpha_{l+1}}(-c_{l+1}).
\end{array}
\end{equation}
On the other hand, by the definition of the transition sequence we have
$$
b_lx_{\gamma_{l+1}(-\alpha_{l+1})}(c_{l+1})\gamma_{l+1}=x_{\gamma_{l+1}(-\alpha_{l+1})}(d_{l+1})\gamma_{l+1}b_{l+1}
=\gamma_{l+1}s_{\alpha_{l+1}}x_{\alpha_{l+1}}(\pm d_{l+1})s_{\alpha_{l+1}}^{-1}b_{l+1}.
$$
Comparing this with~(\ref{eq:p}) and cancelling out $\gamma_{l+1}s_{\alpha_{l+1}}$, we come to a contradiction with
the Bruhat decomposition. Thus the case $u_lc_{l+1}=-1$ is impossible.

Finally, let us look at $b_{j-1}$. Similarly, to our previous calculation the equality
$\bbeta_j(\gamma)=-\alpha=-\bbeta_{(j-1)'}(\gamma)$ implies that
$\gamma_{j-1}\gamma_{j-2}\cdots\gamma_{(j-1)'+1}(\alpha_{(j-1)'})=\gamma_j(-\alpha_j)$.
Hence we get $b_{j-1}=t_{j-1}x_{\gamma_j(-\alpha_j)}(u_{j-1})$. As $b_{j-1}\in B$ and $u_{j-1}\ne0$,
we get $\gamma_j=s_j$. The following calculation:
$$
b_{j-1}\gamma_j=t_{j-1}x_{\alpha_j}(u_{j-1})s_j=x_{\alpha_j}(\alpha_j(t_{j-1})u_{j-1})s_j\cdot s_j^{-1}t_{j-1}s_j
$$
proves that $d_j=\alpha_j(t_{j-1})u_{j-1}\ne0$ (and $b_j=s_j^{-1}t_{j-1}s_j$). We get $d_id_j\ne0$. Similarly
to the case $c_j\ne0$, we get that $\f_i\gamma=[0]^{\f_i\gamma}_s$ does not belong to the closure $\overline{Tz}$.

Let us now suppose that there is no $j>i$ such that $\bbeta_i(\gamma)=-\bbeta_j(\gamma)$.
We take some $c_i\in\C^*$ and consider the point $z=[0,\ldots,0,c_i,0,\ldots,0]^\gamma_s$ (with $c_i$ at the $i$th place).
Taking any root $\tau$ such that $\<\bbeta_i(\gamma),\tau\>>0$, we get
$$
\lim_{c\to0} h_\tau(c)z=[0,\ldots0,c^{\<\bbeta_i(\gamma),\tau\>}c_i,0,\ldots,0]^\gamma=[0]^\gamma_s=\gamma.
$$
We claim that $z\in U^{\f_i\gamma}_s$ and $z=[0,\ldots,0,c_i^{-1},0,\ldots,0]^{\f_i\gamma}_s$.
Let us perform the algorithm of Proposition~\ref{lemma:1}.
It obviously produces $d_1=\cdots=d_{i-1}=0$ and $b_0=\cdots=b_{i-1}=e$.
We will prove by induction on $j=i,\ldots,r$ that $d_i=c_i^{-1}$, $d_{i+1}=\cdots=d_j=0$ and
$b_j=t_jx_{\gamma_j\cdots\gamma_{i+1}\alpha_i}(u_j)$, where $t_j\in T$ and $u_j\ne0$.
The case $j=i$ follows from~(\ref{eq:xminusalpha}) or~(\ref{eq:xalpha}).
Let us suppose that the claim is true for $j<r$ and prove it for $j+1$. As we have
$$
x_{\gamma_j\cdots\gamma_{i+1}\alpha_i}(u_j)\gamma_{j+1}=\gamma_{j+1}x_{\gamma_{j+1}\gamma_j\cdots\gamma_{i+1}\alpha_i}(\pm u_j),
$$
it suffices to prove that $\gamma_{j+1}\gamma_j\cdots\gamma_{i+1}\alpha_i>0$. As we have already seen,
the negativity of the root $\gamma_{j+1}\gamma_j\cdots\gamma_{i+1}\alpha_i$ would imply
$\bbeta_i(\gamma)=-\bbeta_{j+1}(\gamma)$, which is a contradiction.
\end{proof}

\subsection{Topological morphisms} We consider now a special class of morphisms of $\Fold$.

\begin{definition}\label{def:1}
A morphism $(p,w,\phi):s\to s'$ of $\widetilde\Seq$ is called topological if there exists a continuous
$T$-equivariant map $f:\SSigma_s\to\SSigma_{s'}$ such that $f|_{\Gamma_s}=\phi$.
\end{definition}

\begin{definition}\label{def:2} A morphism $(p,w,\phi):s\to s'$ of $\widetilde\Seq$ is called curve preserving if
for any points $\gamma$ and $\delta$ of $\Gamma_s$ connected by a $T$-curve on $\SSigma_s$,
the points $\phi(\gamma)$ and $\phi(\delta)$ are connected by a $T$-curve on $\SSigma_{s'}$.
If this condition is only known to hold for $\delta=\f_i\gamma$, then $(p,w,\phi)$ is called
weakly curve preserving.
\end{definition}

The proof of the following lemma is left to the reader.

\begin{lemma}
The composition of topological {\rm(}resp. curve preserving, weakly curve preserving, with identical rotation, with positive sign{\rm)}
morphisms of $\Fold$ is so. Any identity morphism of $\Fold$ is topological, curve preserving, of identical rotation and
of positive sign.
\end{lemma}

By Proposition~\ref{proposition:T-curve}, any topological morphism is curve preserving and thus weakly curve preserving.
Lemma~\ref{lemma:fcat:6.5} provides us with a combinatorial criterion for a morphism of $\Fold$ to be weakly curve preserving.
We often use the following special case of this criterion.

\begin{lemma}\label{lemma:w}
Let $(p,w,\phi):s\to s'$ be a morphism of $\Fold$, $\gamma\in\Gamma_s$, $i=1,\ldots,|s|$ and $q=p(i)+1,\ldots,|s'|$ be
such that $q\notin\im p$ and $\bbeta_q(\phi(\gamma))=-\bbeta_{p(i)}(\phi(\gamma))$.
If $\bbeta_i(\gamma)=-\bbeta_k(\gamma)$ holds for no $k$ satisfying $p(i)<p(k)<q$,
then the morphism $(p,w,\phi)$ is not weakly curve preserving.
\end{lemma}
\begin{proof}
Let $k$ be the minimal index such that $k>i$ and $\bbeta_i(\gamma)=-\bbeta_k(\gamma)$ or $+\infty$ if there is no such $k$.
We will prove the lemma by backward induction on $k$, starting with the case $k=+\infty$ obvious by~Lemma~\ref{lemma:fcat:6.5}.
Suppose now that $k<+\infty$. 
By the hypothesis of the lemma, we have $p(k)>q$.
Consider the gallery $\widetilde\gamma=\f_k\gamma$. We have
$$
\bbeta_q(\phi(\widetilde\gamma))=
\bbeta_q(\f_{p(k)}\phi(\gamma))=
\bbeta_q(\phi(\gamma))=-\bbeta_{p(i)}(\phi(\gamma))=-\bbeta_{p(i)}(\f_{p(k)}\phi(\gamma))=
-\bbeta_{p(i)}(\phi(\widetilde\gamma)).
$$
Clearly, the minimal index $\tilde k$ such that $\tilde k>i$ and $\bbeta_i(\widetilde\gamma)=-\bbeta_{\tilde k}(\widetilde\gamma)$
(or $+\infty$ if there is no such $\tilde k$) is greater than $k$.
Thus replacing $\gamma$ with $\widetilde\gamma$, we obtain that $(p,w,\phi)$ is not weakly curve preserving by induction.
\end{proof}

Our next aim is to prove a criterion for a morphism of $\Fold$ to be topological. Unfortunately,
our technique based on $T$-curves allows us to prove such a criterion only if
the root system of $G$ is simply laced.

\begin{proposition}\label{proposition:1.0} Let $\lambda,\mu:T\to\C^*$ be characters such that $\ker\lm\subset\ker\mu$.
Then $\mu=n\lm$ for some integer $n$.
\end{proposition}

\begin{lemma}\label{lemma:8}
Let $(p,w,\phi):s\to s'$ be a topological morphism.
Then $(p,e,\phi):s\to s'$ is a morphism of $\Fold$ of positive sign.
\end{lemma}
\begin{proof} We set $r=|\s|$, 
$s_i=s_{\alpha_i}$ and $s'_j=s_{\alpha'_j}$ for simple roots $\alpha_i$ and $\alpha'_j$.
Let $f:\SSigma_s\to\SSigma_{s'}$ be a continuous $T$-equivariant map such that $f|_{\Gamma_s}=\phi$.

Suppose that a gallery $\gamma\in\Gamma_s$ and an index $i=1,\ldots,r$ are
such that $\bbeta_i(\gamma)\ne-\bbeta_j(\gamma)$
for any $j>i$. We set $\widetilde\gamma=\f_i\gamma$, $\delta=\phi(\gamma)$,
$\tilde\delta=\phi(\widetilde\gamma)$ and $j=p(i)$.

Consider a point $z=[0,\ldots,0,c_i,0,\ldots,0]_s^\gamma$ for some $c_i\ne0$ (with $c_i$ at the $i$th place).
The calculations of Lemma~\ref{lemma:fcat:6.5} prove that $z=[0,\ldots,0,c_i^{-1},0,\ldots,0]_s^{\widetilde\gamma}$.
Hence $\gamma,\widetilde\gamma\in\overline{Tz}$.
The set $U^\gamma_s\cap f^{-1}(U^\delta_{s'})$ is a $T$-equivariant open neighbourhood of $\gamma$.
As $\gamma\in\overline{Tz}$, we get $tz\in U^\gamma_s\cap f^{-1}(U^\delta_{s'})$ for some $t\in T$.
By the $T$-equivariance of $U^\gamma_s\cap f^{-1}(U^\delta_{s'})$, we get
$z\in U^\gamma_s\cap f^{-1}(U^\delta_{s'})$. Similar arguments with
$U^{\widetilde\gamma}_s\cap f^{-1}(U^{\tilde\delta}_{s'})$ imply that
$$
z\in U^\gamma_s\cap f^{-1}(U^\delta_{s'})\cap U^{\widetilde\gamma}_s\cap f^{-1}(U^{\tilde\delta}_{s'}).
$$
Let $f(z)=[d]_{s'}^\delta$. By Corollary~\ref{lemma:ww} we get $d_j\ne0$, as $f(z)\in U^{\tilde\delta}_{s'}$.
Hence, we get
$$
\ker\bbeta_i(\gamma)=\Stab_T(z)\subset\Stab_T(f(z))\subset\ker\bbeta_j(\delta).
$$
Proposition~\ref{proposition:1.0} implies $\bbeta_j(\delta)=\epsilon_i\bbeta_i(\gamma)$ for some $\epsilon_i=\pm1$.

Suppose that $\bbeta_j(\delta)=-\bbeta_i(\gamma)$. We get
$
h_{\bbeta_i(\gamma)}(c)z=[0,\ldots,0,c^2c_i,0,\ldots,0]^\gamma,
$
whence
$
\lim_{c\to0}h_{\bbeta_i(\gamma)}(c)z=[0]_s^\gamma=\gamma.
$
On the other hand,
$$
f(h_{\bbeta_i(\gamma)}(c)z)=
h_{\bbeta_i(\gamma)}(c)f(z)=h_{-\bbeta_j(\delta)}(c)f(z)=[\ldots,c^{-2}d_j,\ldots]_s^\delta,
$$
As $d_j\ne0$, the point $f(h_{\bbeta_i(\gamma)}(c)z)$ does not converge to
any point within $U^\delta_{s'}$ including $\delta=f(\gamma)=f(\lim_{c\to0}h_{\bbeta_i(\gamma)}(c)z)$.
This contradicts the continuity of $f$.
We have thus proved $\bbeta_{p(i)}(\phi(\gamma))=\bbeta_i(\gamma)$.

Now let us take an arbitrary gallery $\gamma\in\Gamma_s$ and an index $i=1,\ldots,r$.
There are obviously indices $i<i_1<\cdots<i_k\le r$ such that
$\bbeta_i(\gamma')\ne-\bbeta_j(\gamma')$ for any $j>i$, where $\gamma'=\f_{i_k}\cdots\f_{i_1}\gamma$.
We proved above that $\bbeta_{p(i)}(\phi(\gamma'))=\bbeta_i(\gamma')$.
However, $\bbeta_i(\gamma')=\bbeta(\gamma)$ and
$\bbeta_{p(i)}(\phi(\gamma'))=\bbeta_{p(i)}(\f_{p(i_k)}\cdots\f_{p(i_1)}\phi(\gamma))
=\bbeta_{p(i)}(\phi(\gamma))$, whence $\bbeta_{p(i)}(\phi(\gamma))=\bbeta_i(\gamma)$.
\end{proof}

\begin{lemma}\label{lemma:m} Suppose that the root system of $G$ is simply laced.
Let $s=(s_{\alpha_1},\ldots,s_{\alpha_r})$ and $s'=(s'_{\alpha'_1},\ldots,s_{\alpha'_{r'}})$ be sequences of simple reflections,
where $\alpha_1,\ldots,\alpha_r,\alpha'_1,\ldots,\alpha'_{r'}$ are simple roots.
Let $(p,e,\phi):s\to s'$ be a weakly curve preserving
morphism of $\Fold$ of positive sign. Let $[c]_s^\gamma=[\tilde c]_s^{\f_i\gamma}$ be a point of
the intersection $U_s^\gamma\cap U_s^{\f_i\gamma}$. We define
$$
d_{p(k)}=\sigma_{\alpha'_{p(k)}}^{\phi(\gamma),p(k)}\sigma_{\alpha_k}^{\gamma,k}c_k,\quad
\tilde d_{p(k)}=\sigma_{\alpha'_{p(k)}}^{\phi(\f_i\gamma),p(k)}\sigma_{\alpha_k}^{\f_i\gamma,k}\tilde c_k
$$
for $k=1,\ldots,r$ and $d_l=\tilde d_l=0$ for $l\in[1,r']\setminus \im p$. Then we have
$[d]_{s'}^{\phi(\gamma)}=[\tilde d]_{s'}^{\phi(\f_i\gamma)}$.
\end{lemma}
\begin{proof}
We set $\delta=\phi(\gamma)$, $\widetilde\gamma=\f_i\gamma$ and $\tilde\delta=\phi(\f_i\gamma)$. 
Let $(b_1,\ldots,b_r)$ be the transition sequence of $[c]_s^\gamma$ from $U^\gamma_s$ to $U^{\widetilde\gamma}_s$.
We define the sequence $(b'_1,\ldots,b'_{r'})$ of elements of $G$ inductively as follows:
$$
b_0=e,\qquad b'_lx_{\delta_{l+1}(-\alpha'_{l+1})}(d_{l+1})\delta_{l+1}=x_{\tilde\delta_{l+1}(-\alpha'_{l+1})}(\tilde d_{l+1})\tilde\delta_{l+1}b'_{l+1}.
$$
Noninductively these elemts are given by
\begin{equation}\label{eq:19}
b'_{l+1}=(\llbracket \tilde d_1,\ldots,\tilde d_{l+1}\rrbracket^{\tilde\delta}_{s'})^{-1}\llbracket d_1,\ldots,d_{l+1}\rrbracket^\delta_{s'}.
\end{equation}
The definition of elements $b'_l$ resembles that of Proposition~\ref{lemma:1}.
To prove that $(b_1,\ldots,b'_{r'})$ is really the transition sequence of $[d]_{s'}^{\delta}$
form $U_{s'}^\delta$ to $U_{s'}^{\tilde\delta}$, it suffuces to prove that $b'_1,\ldots,b'_{r'}\in B$.
The required equality $[d]_{s'}^{\delta}=[\tilde d]_{s'}^{\tilde\delta}$ will then immediately
follow from Proposition~\ref{lemma:1}.

We will prove inductively on $l=p(i),\ldots,r'$ the following stronger assertion:
\medskip
{\renewcommand{\labelenumi}{$(\dagger)$}
\renewcommand{\theenumi}{$(\dagger)$}
\begin{enumerate}
\item\label{lemma:m:pp:1} $b'_l$ is a product of elements of $T$ and
      the root elements $x_{\tau_l\cdots\tau_{p(i)+1}\alpha'_{p(i)}}(a)$ such that
      $\tau_k\in\{e,s'_k\}$, $\tau_k=\delta_k$ for $k\notin\im p$ and $\tau_l\cdots\tau_{p(i)+1}\alpha'_{p(i)}>0$.
\end{enumerate}}
By Corollary~\ref{lemma:ww}, we have $c_i=\tilde c_i^{-1}$.
Moreover, $\sigma_{\alpha_i}^{\gamma,i}=\sigma_{\alpha_i}^{\widetilde\gamma,i}$ and $\sigma_{\alpha'_{p(i)}}^{\delta,p(i)}=\sigma_{\alpha'_{p(i)}}^{\tilde\delta,p(i)}$.
Hence $\tilde d_{p(i)}=d_{p(i)}^{-1}$ and~\ref{lemma:m:pp:1} for $l=p(i)$ follows from~(\ref{eq:xminusalpha}) and~(\ref{eq:xalpha}).

Suppose now that $p(i)\le l<r'$ and that~\ref{lemma:m:pp:1} holds for $l$.
We will prove this assertion for $l+1$. First consider the case $l+1\notin\im p$.
We have $d_{l+1}=\tilde d_{l+1}=0$ by definition, whence
$$
b'_{l+1}=\delta_{l+1}^{-1}b'_l\delta_{l+1}.
$$
If $\delta_{l+1}=e$, then $b'_{l+1}=b'_l\in B$. This $b'_{l+1}$ has the form required in~\ref{lemma:m:pp:1}.
Suppose now that $\delta_{l+1}=s'_{l+1}$. If $\tau_l\cdots\tau_{p(i)+1}\alpha'_{p(i)}$ is a root as in~\ref{lemma:m:pp:1}
different from $\alpha'_{l+1}$, then we get
$$
(s'_{l+1})^{-1}x_{\tau_l\cdots\tau_{p(i)+1}\alpha'_{p(i)}}(a)s'_{l+1}=x_{\delta_{l+1}\tau_l\cdots\tau_{p(i)+1}\alpha'_{p(i)}}(\pm a),
$$
where $\delta_{l+1}\tau_l\cdots\tau_{p(i)+1}\alpha'_{p(i)}>0$.

Suppose now that
\begin{equation}\label{eq:m:15}
\tau_l\cdots\tau_{p(i)+1}\alpha'_{p(i)}=\alpha'_{l+1}
\end{equation}
for some root $\tau_l\cdots\tau_{p(i)+1}\alpha'_{p(i)}$ as in~\ref{lemma:m:pp:1}.
We are going to prove that this case is impossible. We define the gallery $\tau\in\Gamma_{s'}$
by completing the already existing values $\tau_{p(i)+1},\ldots,\tau_l$ by
$\tau_m=\delta_m$ for $m\in[1,r']\setminus(p(i),l]$.
We set $\beta=\bbeta_{l+1}(\tau)$ and write~(\ref{eq:m:15}) as follows
\begin{equation}\label{eq:m:16}
\bbeta_{p(i)}(\tau)=-\bbeta_{l+1}(\tau)=-\beta.
\end{equation}
Let $j$ be the maximal index of $\im p\cap[p(i),l+1]$ such that $\bbeta_j(\tau)=\pm\beta$. Note that $j$ is well-defined
by~(\ref{eq:m:16}).

We denote by $q$ the minimal index such that $q>j$ and the cambers $\tau^jC$ and $\tau^qC$ are separated by $L_\beta$.
As $\tau_{l+1}=s'_{l+1}$, the gallery $\tau^jC,\ldots,\tau^{l+1}C$ crosses
the hyperplane $L_\beta$, whence $q$ is well-defined and $q\le l+1$. We get
\begin{equation}\label{eq:17}
\bbeta_q(\tau)=-\bbeta_j(\tau)=\pm\beta.
\end{equation}
Hence $q\notin\im p$ by the choice of $j$.

Let us come back to $\Gamma_s$. We set $\xi=\phi^{-1}(\tau)$ and $u=p^{-1}(j)$. If $\bbeta_v(\xi)=\pm\bbeta_u(\xi)$
for some $v>u$ such that $p(v)\le l+1$, then we get
$$
\pm\beta=\pm\bbeta_j(\tau)=\pm\bbeta_{p(u)}(\phi(\xi))=\bbeta_u(\xi)=\pm\bbeta_v(\xi)=\pm\bbeta_{p(v)}(\phi(\xi))=\pm\bbeta_{p(v)}(\tau).
$$
This equality can not hold by the choice of $j$. Applying Lemma~\ref{lemma:w}, we get a contradiction with
the hypothesis that $(p,e,\phi)$ is weakly curve preserving.

Finally, consider the case $l+1\in\im p$. We set $k=p^{-1}(l+1)$. We have $k>i$.
The right-hand side of~(\ref{eq:19}) can be written as follows:
\begin{multline*}
b'_{l+1}=(\tilde\delta^{p(k)})^{-1}x_{\bbeta_{p(k)}(\tilde\delta)}(-\sigma_{\alpha'_{p(k)}}^{\tilde\delta,p(k)}\tilde d_{p(k)})
\cdots x_{\bbeta_{p(1)}(\tilde\delta)}(-\sigma_{\alpha'_{p(1)}}^{\tilde\delta,p(1)}\tilde d_{p(1)})\times\\
\shoveright{\times x_{\bbeta_{p(1)}(\delta)}(\sigma_{\alpha'_{p(1)}}^{\delta,p(1)}d_{p(1)})\cdots x_{\bbeta_{p(k)}(\delta)}(\sigma_{\alpha'_{p(k)}}^{\delta,p(k)}d_{p(k)})\delta^{p(k)}}\\
\shoveleft{=(\tilde\delta^{p(k)})^{-1}\widetilde\gamma^k\cdot(\widetilde\gamma^k)^{-1}x_{\bbeta_k(\widetilde\gamma)}(-\sigma_{\alpha_k}^{\widetilde\gamma,k}\tilde c_k)
\cdots x_{\bbeta_1(\widetilde\gamma)}(-\sigma_{\alpha_1}^{\widetilde\gamma,1}\tilde c_1)\times}\\
\shoveright{\times x_{\bbeta_1(\gamma)}(\sigma_{\alpha_1}^{\gamma,1}c_1)\cdots x_{\bbeta_k(\gamma)}(\sigma_{\alpha_k}^{\gamma,k}c_k)\gamma^k\cdot(\gamma^k)^{-1}\delta^{p(k)}=(\tilde\delta^{p(k)})^{-1}\widetilde\gamma^k b_k (\gamma^k)^{-1}\delta^{p(k)}.}\\
\end{multline*}
To prove that $b'_{l+1}\in B$, we do the following calculation in $W$:
$$
(\tilde\delta^{p(k)})^{-1}\widetilde\gamma^k=(s_{\bbeta_{p(i)}(\delta)}^{\cond{p(i)\le p(k)}}\delta^{p(k)})^{-1}s_{\bbeta_i(\gamma)}^{\cond{i\le k}}\gamma^k
=(\delta^{p(k)})^{-1}\gamma^k.
$$
Hence in $G$, we get
$
b'_{l+1}=t(\delta^{p(k)})^{-1}\gamma^k b_k (\gamma^k)^{-1}\delta^{p(k)}
$
for some $t\in T$. So it suffices to prove that
$
(\delta^{p(k)})^{-1}\gamma^k\tau_k\cdots\tau_{i+1}\alpha_i>0
$, where $\tau_{i+1},\ldots,\tau_k$ are as in Corollary~\ref{corollary:1}. We define a gallery $\tau\in\Gamma_s$ by
completing the already existing values $\tau_{i+1},\ldots,\tau_k$ by $\tau_m=\gamma_m$ for $m\in[1,r]\setminus(i,k]$.
Therefore, we have $\tau=\f_{i_1}\cdots\f_{i_q}\gamma$ for some $i_1,\ldots,i_q\in(i,k]$, whence
\begin{multline*}
(\delta^{p(k)})^{-1}\gamma^k\tau_k\cdots\tau_{i+1}\alpha_i=(\delta^{p(k)})^{-1}\gamma^k(\tau^k)^{-1}\tau^i\alpha_i
=-(\delta^{p(k)})^{-1}\gamma^k((\f_{i_1}\cdots\f_{i_q}\gamma)^k)^{-1}\bbeta_i(\tau)\\
=-(\delta^{p(k)})^{-1}\gamma^k(s_{\bbeta_{i_1}(\f_{i_2}\cdots\f_{i_q}\gamma)}s_{\bbeta_{i_2}(\f_{i_3}\cdots\f_{i_q}\gamma)}\cdots s_{\bbeta_{i_q}(\gamma)}\gamma^k)^{-1}\bbeta_i(\tau)\\
=-(s_{\bbeta_{p(i_1)}(\f_{p(i_2)}\cdots\f_{p(i_q)}\delta)}s_{\bbeta_{p(i_2)}(\f_{p(i_3)}\cdots\f_{p(i_q)}\delta)}\cdots s_{\bbeta_{p(i_q)}(\delta)}\delta^{p(k)})^{-1}\bbeta_{p(i)}(\phi(\tau))\\
=-((\f_{p(i_1)}\cdots\f_{p(i_q)}\delta)^{p(k)})^{-1}\bbeta_{p(i)}(\phi(\tau))
=-(\phi(\f_{i_1}\cdots\f_{i_q}\gamma)^{p(k)})^{-1}\bbeta_{p(i)}(\phi(\tau))\\
=-(\phi(\tau)^{p(k)})^{-1}\bbeta_{p(i)}(\phi(\tau)).
\end{multline*}
By Corollary~\ref{corollary:1}, we have
$-(\tau^k)^{-1}\bbeta_i(\tau)=\tau_k\cdots\tau_{i+1}\alpha_i>0$.
Thus it suffices to prove the implication
$$
(\tau^k)^{-1}\bbeta_i(\tau)<0\Rightarrow(\phi(\tau)^{p(k)})^{-1}\bbeta_{p(i)}(\phi(\tau))<0
$$
for any $\tau\in\Gamma_s$. As $(\tau^i)^{-1}\bbeta_i(\tau)=-\alpha_i<0$, the left-hand side is equivalent to
the fact that $L_{\bbeta_i(\tau)}$ does not separate $\tau^iC$ and $\tau^kC$.
Interpreting similarly the right-hand side, we get an equivalent implication
$$
L_\beta\text{ does not separate }\tau^iC\text{ and }\tau^kC\Rightarrow L_\beta\text{ does not separate }\phi(\tau)^{p(i)}C\text{ and }\phi(\tau)^{p(k)}C,
$$
where $\beta=\bbeta_i(\tau)=\bbeta_{p(i)}(\phi(\tau))$.

We will prove this implication inductively on $k-i$, the case $k=i$ being trivial. If $\bbeta_v(\tau)=\beta$ for some
$v\in(i,k]$, then $\bbeta_{p(v)}(\phi(\tau))=\beta$. By Proposition~\ref{proposition:2}, the hyperplane $L_\beta$ separates neither $\tau^iC$ and $\tau^vC$
nor $\phi(\tau)^{p(i)}C$ and $\phi(\tau)^{p(v)}C$. Thus the implication for the pair $(i,k)$ follows from
the implication for the pair $(v,k)$, which is known to be true by the inductive hypothesis.
Therefore, we will suppose in the rest of the proof that there is no such $v$.

Suppose that $L_\beta$ does not separate $\tau^iC$ and $\tau^kC$ but separates $\phi(\tau)^{p(i)}C$ and $\phi(\tau)^{p(k)}C$.
If $\bbeta_k(\tau)=\pm\beta$ then we get $\bbeta_k(\tau)=\beta$ by Proposition~\ref{proposition:2}.
This case is handled by the previous paragraph.

Therefore, we suppose that $\bbeta_k(\tau)\ne\pm\beta$. In this case, $L_\beta$ does not separate $\tau^{k-1}C$ and $\tau^kC$.
Hence $L_\beta$ also does not separate $\tau^iC$ and $\tau^{k-1}C$. By the inductive hypothesis,
$L_\beta$ does not separate $\phi(\tau)^{p(i)}C$ and $\phi(\tau)^{p(k-1)}C$.
Therefore $L_\beta$ separates $\phi(\tau)^{p(k-1)}C$ and $\phi(\tau)^{p(k)}C$, whence there exists some
$q\in(p(k-1),p(k))$ such that $\bbeta_q(\phi(\tau))=-\beta$.


Let $v$ be the maximal element of $[i,k]$ such that $\bbeta_v(\tau)=\pm\beta$.
As $L_\beta$ does not separate $\tau^iC$ and $\tau^kC$, we get $\bbeta_v(\tau)=\beta$.
Our stipulation above implies that $v=i<k$. By Lemma~\ref{lemma:w}, the morphism $(p,e,\phi)$ is not
weakly curve preserving. This contradicts the hypothesis of the lemma. This proves the implication
and that $b'_{l+1}\in B$.
Lemma~\ref{lemma:x:1} implies now that~\ref{lemma:m:pp:1} holds for $b'_{l+1}$.
\end{proof}

\begin{theorem}\label{theorem:1}
Let $(p,w,\phi):s\to s'$ be a morphism of $\Fold$. Consider the following conditions:
\begin{enumerate}
\item\label{theorem:1:cond:1} $(p,w,\phi)$ is topological;
\item\label{theorem:1:cond:2} $(p,e,\phi):s\to s'$ is a weakly curve preserving morphism of $\Fold$ of positive sign;
\end{enumerate}
For any of root system of $G$,~\ref{theorem:1:cond:1} implies~\ref{theorem:1:cond:2}. If the root system of $G$ is simply laced,
then~\ref{theorem:1:cond:2} implies~\ref{theorem:1:cond:1}.
\end{theorem}
\begin{proof}  Suppose that~\ref{theorem:1:cond:1} holds. By Proposition~\ref{lemma:8}, the triple $(p,e,\phi)$ is a morphism of
$\Fold$ of positive sign. By Proposition~\ref{proposition:T-curve}, the morphism $(p,e,\phi)$ is curve preserving,
thus also weakly curve preserving.

Suppose now that the root system of $G$ is simply laced and~\ref{theorem:1:cond:2} holds.
We set $r=|s|$ and $r'=|s'|$ for brevity.
For any $\gamma\in\Gamma_s$, we define the map $\psi^\gamma:U^\gamma_s\to U^{\phi(\gamma)}_{s'}\hookrightarrow\SSigma_{s'}$ by
$$
\psi^\gamma([c_1,\ldots,c_r]^\gamma)=[d_1,\ldots,d_{r'}]^{\phi(\gamma)},\text{ where }
d_j=\left\{
\begin{array}{ll}
\sigma_{\alpha'_{p(k)}}^{\phi(\gamma),p(k)}\sigma_{\alpha_k}^{\gamma,k}c_k&\text{ if }j=p(k);\\[6pt]
0&\text{ if }j\notin\im p.
\end{array}
\right.
$$
This map is obviously a $T$-equivariant morphism of algebraic varieties. 
By Lemma~\ref{lemma:m}, the morphisms $\psi^\gamma$ and $\psi^{\f_i\gamma}$ coincide on the intersection
$U^\gamma_s\cap U^{\f_i\gamma}_s$. Therefore all morphisms $\psi^\gamma$ coincide on
$U=\bigcap_{\gamma\in\Gamma_s} U_s^\gamma$. This intersection is nonempty, as $\SSigma_s$ is irreducible.
Let us take any two galleries $\gamma,\delta\in\Gamma_s$.
The intersection $U_s^\gamma\cap U_s^\delta$ is irreducible in the Zarisski topology.
The restrictions of $\psi^\gamma$ and $\psi^\delta$ to this set coincide on a Zarisski closed subset, which contains~$U$.
By irreducibility the morphisms $\psi^\gamma$ and $\psi^\delta$ coincide on $U_s^\gamma\cap U_s^\delta$.
Therefore we can glue all morphisms $\psi^\gamma$ to one $T$-equivariant morphism $\psi$ of varieties (see~\cite[Proposition~2.3]{Hum}
about the affine criterion). Clearly, $\psi|_{\Gamma_s}=\phi$.
\end{proof}

\subsection{Intermediate categories} We are now ready to answer the question raised in the introduction about
the existence of a functor $\BS':\Seq'\to\Top(T)$ that makes Diagram~(\ref{eq:z}) commutative. We begin with the following
easy observation.

\begin{lemma}\label{lemma:z}
Let $c$ be a nonzero element of $S$ and $w$ be a nonidentical element of $W$. Then there exists $b\in S$ such that $w(bc)\ne(bc)$.
\end{lemma}
\begin{proof}
If $wc\ne c$, then we can take $b=1$. Suppose that on the contrary $wc=c$. As $w\ne e$, there exists a root $\alpha\in R$
such that $w\alpha\ne\alpha$. Suppose that $w(\alpha c)=\alpha c$.
Then we get
$$
\alpha c=w(\alpha c)=(w\alpha)(wc)=(w\alpha)c.
$$
As there are no zero devisors in $S$, cancelling out $c\ne0$, we get a contradiction $\alpha=w\alpha$.
Hence we take $b=\alpha$ in this case.
\end{proof}

\begin{theorem}\label{theorem:cat}
%
Let $\Seq'$ be an intermediate category between $\Seq$ and $\Fold$.
Consider the following conditions:
\begin{enumerate}
\item\label{theorem:cat:cond:1} A functor $\BS'$ making Diagram~(\ref{eq:z}) commutative exists;
\item\label{theorem:cat:cond:2} All morphisms of the category $\Seq'$ are topological and of identical rotation
\end{enumerate}
For any of root system of $G$,~\ref{theorem:cat:cond:1} implies~\ref{theorem:cat:cond:2}.
If the root system of $G$ is simply laced,
then~\ref{theorem:cat:cond:2} implies~\ref{theorem:cat:cond:1}.

\end{theorem}
\begin{proof}
First suppose that~\ref{theorem:cat:cond:1} holds.
As $\Obj(\Seq')=\Obj(\Seq)$, we get $\BS'(s)=\SSigma_s$ for any object $s$.

Let $(p,w,\phi):s\to s'$ be a morphism of $\Seq'$.
Consider the $T$-equivariant continuous map $f=\BS'((p,w,\phi)):\SSigma_s\to\SSigma_{s'}$.
We claim that $f|_{\Gamma_s}=\phi$.

We choose $v\in\X(s')$ so that $v(\phi(\gamma))\ne0$ but $v(\delta)=0$
for any $\delta\ne\phi(\gamma)$. This element can be written, for example, as follows:
$v=\nabla_{\phi(\gamma)_{|s'|}}\cdots\nabla_{\phi(\gamma)_2}\nabla_{\phi(\gamma)_1}(1)$,
where $\nabla_?$ is the operator of concentration defined in~\cite[Section 4.2]{scbs}.
By the commutativity of Diagram~(\ref{eq:z}), we get $f^*=\widetilde H((p,w,\phi))$.
Hence and from Diagram~(\ref{eq:zz}), we get the following commutative diagram:
$$
\begin{tikzcd}[column sep=6em]
H_T^\bullet(\SSigma_{s'})\arrow{r}{f^*}\arrow[hook]{d}&H_T^\bullet(\SSigma_{s})\arrow[hook]{d}\\
H_T^\bullet(\Gamma_{s'})\arrow{r}{?_{(p,w,\phi)}}&H_T^\bullet(\Gamma_{s})
\end{tikzcd}\qquad
$$
where the vertical arrow are the natural restrictions.

Let $u$ be the element of the upper left corner of this diagram that is mapped to $v$ by the left vertical arrow.
Starting from $u$ and going along the lower path, we get $w^{-1}(v\circ\phi(?))$ in the lower right corner.
Along the upper path, we get $f^*(u)$. To calculate this function, consider the following commutative diagram
of $T$-equivariant commutative maps:
$$
\begin{tikzcd}
\SSigma_s\arrow{r}{f}&\SSigma_{s'}\\
\Gamma_s\arrow{r}{f|_{\Gamma_s}}\arrow[hook]{u}&\Gamma_{s'}\arrow[hook]{u}
\end{tikzcd}
$$
Applying $H^\bullet_T$, we get the commutative diagram of cohomologies
$$
\begin{tikzcd}
H^\bullet_T(\SSigma_s')\arrow{r}{f^*}\arrow[hook]{d}&H^\bullet_T(\SSigma_s)\arrow[hook]{d}\\
H^\bullet_T(\Gamma_s')\arrow{r}{(f|_{\Gamma_s})^*}&H^\bullet_T(\Gamma_s)
\end{tikzcd}
$$
Hence $w^{-1}(v\circ\phi(?))=f^*(u)=(f|_{\Gamma_s})^*(v)=v\circ f$.
Evaluating both sides at $\gamma$, we get
\begin{equation}\label{eq:0}
w^{-1}(v\circ\phi(\gamma))=v\circ f(\gamma)
\end{equation}
Thus $\phi(\gamma)=f(\gamma)$, as otherwise we would get a contradiction $0\ne w^{-1}(v\circ\phi(\gamma))=v\circ f(\gamma)=0$.
So we get from~(\ref{eq:0}) that $v\circ\phi(\gamma)$ is $w$-stable.
However, we can replace $v$ by any product $bv$, where $b\in S\setminus 0$. If $w\ne e$, then Lemma~\ref{lemma:z} proves that
we can choose $v$ so that $v\circ\phi(\gamma)$ is not $w$-stable.
This contradiction proves that $w=e$.

Suppose now that the root system of $G$ is simply laced and~\ref{theorem:cat:cond:2} holds.
We are going to construct the functor $\BS'$ as follows.
For any sequence $s$ of simple roots, we set $\BS'(s)=\SSigma_s$ and for any morphism
$(p,e,\phi)$ of $\Fold$, we define $\BS'\big((p,e,\phi)\big)$ to be the morphism $\psi$ defined in the proof of Theorem~\ref{theorem:1}.

First, we need to check that $\BS'$ so defined is indeed a functor. If $(p,e,\phi)$ is an identity morphism, then it is obvious
that each map $\psi^\gamma$ defined in the proof of Theorem~\ref{theorem:1} is the identity map from $U^\gamma_s$ to itself.
Therefore their gluing $\psi$ is the identity map on $\SSigma_s$.

We consider now the composition of morphisms $(p,e,\phi):s\to s'$ and $(p',e,\phi'):s'\to s''$ of $\Seq'$.
Let $\psi^\gamma:U^\gamma_s\to U^{\phi(\gamma)}_{s'}$ be the maps for $(p,e,\phi)$,
$(\psi')^\delta:U^\delta_{s'}\to U^{\phi'(\delta)}_{s''}$ be the maps for $(p',e,\phi')$
and $\bar\psi^\gamma:U^\gamma_{s}\to U^{\phi'\phi(\gamma)}_{s''}$ be the maps for the
composition $(p'p,e,\phi'\phi)$ constructed
as in the proof of Theorem~\ref{theorem:1}. We claim that $\bar\psi^\gamma=(\psi')^{\phi(\gamma)}\psi^\gamma$.

Indeed, take a point $[c]^\gamma_s$ of $U^\gamma_s$. Let $\psi^\gamma([c]^\gamma_s)=[c']^{\phi(\gamma)}_{s'}$ and
$\psi^\gamma([c']^{\phi(\gamma)}_{s'})=[c'']^{\phi'\phi(\gamma)}_{s''}$ for the corresponding coordinates $c'_j$ and $c''_k$,
which are calculated by
$$
c'_j=\left\{
\begin{array}{ll}
\sigma_{\alpha'_{p(i)}}^{\phi(\gamma),p(i)}\sigma_{\alpha_i}^{\gamma,i}c_i&\text{ if }j=p(i);\\[6pt]
0&\text{ if }j\notin\im p,
\end{array}
\right.\qquad
c''_k=\left\{
\begin{array}{ll}
\sigma_{\alpha''_{p'(j)}}^{{\phi'\phi}(\gamma),p'(j)}\sigma_{\alpha'_j}^{\phi(\gamma),j}c'_j&\text{ if }k=p'(j);\\[6pt]
0&\text{ if }k\notin\im p',
\end{array}
\right.
$$
where $s_i=s_{\alpha_i}$, $s'_j=s_{\alpha'_j}$, $s''_k=s_{\alpha''_k}$ and the roots $\alpha_i,\alpha'_j,\alpha''_k$ are simple.
We have to calculate $c''_k$ via $c_i$.
From the above formulas, it is clear that $c''_k=0$ unless $k\in\im p'p$. In the last case, we denote $j=(p')^{-1}(k)$
and $i=p^{-1}(j)$. We have
$$
c''_k=\sigma_{\alpha''_{p'(j)}}^{{\phi'\phi}(\gamma),p'(j)}\sigma_{\alpha'_j}^{\phi(\gamma),j}c'_j
=\sigma_{\alpha''_{p'(j)}}^{{\phi'\phi}(\gamma),p'(j)}\sigma_{\alpha'_j}^{\phi(\gamma),j} \sigma_{\alpha'_{p(i)}}^{\phi(\gamma),p(i)}\sigma_{\alpha_i}^{\gamma,i}c_i
=\sigma_{\alpha''_{p'p(i)}}^{{\phi'\phi}(\gamma),p'p(i)}\sigma_{\alpha_i}^{\gamma,i}c_i,
$$
as required.

Let us check that Diagram~(\ref{eq:z}) is commutative. It is clearly so on the level of objects.
To check its commutativity on the level of morphisms, we first consider the left triangle of Diagram~(\ref{eq:z})
containing the dashed arrow. Let $p:s\to s'$ be a morphism of $\Seq$. By definition, it is a monotone map from $[1,r]$ to $[1,r']$
such that $\alpha_k=\alpha'_{p(k)}$, where $r=|s|$, $r'=|s'|$ and $s_k=s_{\alpha_k}$, $s'_j=s_{\alpha'_j}$ for
$\alpha_k,\alpha'_j\in\Pi$.

We have to restrict the morphism $\BS(p)$
to each chart $U_s^\gamma$ and prove that it coincides with the morphism $\psi^\gamma:U_s^\gamma\to U_{s'}^{\phi(\gamma)}$ defined
by $(p,e,\phi)$ as in Theorem~\ref{theorem:1}, where $\phi=\phi^p$ (see Section~\ref{widetildeH}). 
For any point $[c]^\gamma_s\in U^\gamma_s$, we get
\begin{multline*}
\BS(p)([c]^\gamma_s)
=[e,\ldots,e,\underbrace{x_{\gamma_1(-\alpha_1)}(c_1)\gamma_1}_{p(1)\text{th place}},e\ldots,e,\underbrace{x_{\gamma_r(-\alpha_r)}(c_r)\gamma_r}_{p(r)\text{th place}},e\ldots,e]\\
=[e,\ldots,e,\underbrace{x_{\phi(\gamma)_{p(1)}(-\alpha'_{p(1)})}(c_1)\phi(\gamma)_{p(1)}}_{p(1)\text{th place}},e\ldots,e,\underbrace{x_{\phi(\gamma)_{p(r)}(-\alpha'_{p(r)})}(c_r)\phi(\gamma)_{p(r)}}_{p(r)\text{th place}},e\ldots,e]
=[c']^{\phi(\gamma)}_{s'},
\end{multline*}
where
$$
c'_j=\left\{
\begin{array}{ll}
c_k&\text{ if }j=p(k);\\[6pt]
0&\text{ if }j\notin\im p.
\end{array}
\right.
$$
Comparing this formula with the one that defines $\psi^\gamma$ in the proof of Theorem~\ref{theorem:1}, we conclude that
it suffices to prove $\sigma_{\alpha'_{p(i)}}^{\phi(\gamma),p(i)}=\sigma_{\alpha_i}^{\gamma,i}$. This is however
clear from the definitions of $\sigma^{\gamma,i}_\alpha$ and $\phi=\phi^p$.

Finally, let us consider the right triangle of Diagram~(\ref{eq:z}) containing the dashed arrow.
We take any morphism $(p,e,\phi):s\to s'$ of $\Seq'$ and will prove that $\widetilde H((p,e,\phi))=\psi^*$,
where $\psi=\BS'((p,e,\phi))$. 
By (\ref{eq:zz}), we need to prove that the diagram
$$
\begin{tikzcd}[column sep=6em]
H_T^\bullet(\SSigma_{s'})\arrow{r}{\psi^*}\arrow[hook]{d}&H_T^\bullet(\SSigma_s)\arrow[hook]{d}\\
H_T^\bullet(\Gamma_{s'})\arrow{r}{?_{(p,e,\phi)}}&H_T^\bullet(\Gamma_s)
\end{tikzcd}
$$
is commutative. We have to prove that
$$
\psi^*(g)(\gamma)=(g|_{\Gamma_{s'}})_{(p,e,\phi)}(\gamma)=g(\phi(\gamma))
$$
for any $g\in H_T^\bullet(\SSigma_{s'})$ and $\gamma\in\Gamma_s$. This fact follows from the commutative diagram
$$
\begin{tikzcd}[column sep=6em]
H_T^\bullet(\SSigma_{s'})\arrow{r}{\psi^*}\arrow{d}&H_T^\bullet(\SSigma_s)\arrow{d}\\
H_T^\bullet(\{\phi(\gamma)\})\arrow[equal]{r}&H_T^\bullet(\{\gamma\})
\end{tikzcd}
$$
Here we identify the $T$-equivariant cohomologies of (different) points, as stipulated
in the introduction.
\end{proof}


\section{Appendix}

\subsection{Examples}

All examples below are for type $A_n$ with the simple roots
$\alpha_1,\ldots,\alpha_{n-1}$ ordered as in the following Dynkin diagram:

\begin{center}
\setlength{\unitlength}{1.2mm}
\begin{picture}(45,0)
\put(0,0){\circle{1}}
\put(10,0){\circle{1}}
\put(20,0){\circle{1}}
\put(40,0){\circle{1}}
\put(50,0){\circle{1}}
\put(0.5,0){\line(1,0){9}}
\put(10.5,0){\line(1,0){9}}
\put(40.5,0){\line(1,0){9}}
\put(20.5,0){\line(1,0){4}}
\put(39.5,0){\line(-1,0){4}}

\put(27.25,0){\circle{0}}
\put(30,0){\circle{0}}
\put(32.75,0){\circle{0}}

\put(-1,-3.5){$\scriptstyle\alpha_1$}
\put(9,-3.5){$\scriptstyle\alpha_2$}
\put(19,-3.5){$\scriptstyle\alpha_3$}
\put(49,-3.5){$\scriptstyle\alpha_{n-1}$}
\put(39,-3.5){$\scriptstyle\alpha_{n-2}$}
\end{picture}
\end{center}

\hspace{20mm}

\noindent
The simple reflections are $s_i=s_{\alpha_i}$.

\begin{example}\label{Ex:3}{\rm By definition, for any morphism $p:s\to s'$ of $\Seq$, we have $s'_{p(i)}=s_i$.
This is not in general true for morphisms of $\Fold$, as is shown by the following example: $s=(s_1,s_2)$,
$s'=(s_1,s_2,s_1)$, $p(1)=1$, $p(2)=3$ and $\phi$ is given by
$$
(e  ,e )\stackrel\phi\mapsto(s_1,s_2,e),\quad
(s_1,e)\stackrel\phi\mapsto(e,s_2,e),\quad
(e  ,s_2)\stackrel\phi\mapsto(s_1,s_2,s_1),\quad
(s_1,s_2)\stackrel\phi\mapsto(e,s_2,s_1).
$$
The triple $(p,e,\phi):s\to s'$ is a morphism of $\Fold$ having identical rotation and sign $(-1,1)$.
}
\end{example}

In the following examples, the maps between sets of fixed points are given only for one gallery,
as the remaining values can be reconstructed by condition~\ref{F:3} of Section~\ref{morph}.

\begin{example}\label{Ex:4}{\rm Let $s=(s_2,s_1)$, $s'=(s_1,s_2,s_1)$, $p(1)=1$, $p(2)=3$, $\phi((e,e))=(s_1,s_2,e)$
and $w=s_1s_2s_1$. Then $(p,w,\phi):s\to s'$ is a morphism of $\Fold$ of sign $(1,-1)$.}\end{example}

\begin{example}\label{Ex:5}{\rm Let $s=(s_1,s_2,s_3)$, $s'=(s_1,s_2,s_1,s_3,s_2)$, $p(1)=1$, $p(2)=3$, $p(3)=5$
and $\phi((e,e,e))=(s_1,s_2,e,s_3,e)$. Then $(p,e,\phi)$ is a morphism of $\Fold$ of sign $(-1,1,1)$.}
\end{example}

\begin{example}\label{Ex:6}{\rm Let $s=(s_2,s_3,s_1)$, $s'=(s_3,s_4,s_2,s_4,s_4)$, $p(1)=1$, $p(2)=3$, $p(3)=5$,
$w=(1,5)(2,3,4)$ and $\phi((e,e,e))=(s_3,s_4,s_2,s_4,e)$.
Then $(p,w,\phi)$ is a morphism of $\Fold$ of sign $(-1,1,-1)$.}
\end{example}

\begin{example}\label{Ex:7}{\rm Let $s=(s_4, s_3, s_3)$, $s'=(s_4, s_3, s_2, s_3, s_2, s_1)$,
$p(1)=1$, $p(2)=2$, $p(3)=5$ and $\phi((e,e,e))=(e, e, s_2, s_3, e, s_1)$.
All pairs $\gamma$ and $\f_i\gamma$ of $T$-fixed points of $\BS(s)$ are connected by $T$-curves
except the following two: $(e, e, s_3)$ and $(e, s_3, s_3)$; $(s_4, e, s_3)$ and $(s_4, s_3, s_3)$.
The triple $(p,e,\phi)$ is a topological morphism of $\Fold$.}
\end{example}

Note that the morphisms in Examples~\ref{Ex:3}--\ref{Ex:6} are not topological either because
of the wrong sign or because the triples $(p,e,\phi)$ are not morphisms of $\Fold$ (if $w\ne e$).
The following example shows that a morphism of $\Fold$ can be not topological because
it does not preserve $T$-curves.

\begin{example}\label{Ex:8}{\rm Let $s=(s_1,s_4,s_3)$, $s'=(s_1, s_4, s_4, s_1, s_3, s_4)$,
$p(1)=1$, $p(2)=3$, $p(3)=5$ and $\phi((e,e,e))=(e, e, e, s_1, e, s_4)$. Then
$(p,e,\phi)$ is a morphism of $\Fold$ of sign $(1,1,1)$, which is not topological.
Actually, this morphism is not weakly curve preserving as the points $\gamma$ and $\delta=\f_i\gamma$
are connected by a $T$-curve on $\BS(s)$ but the points $\phi(\gamma)$ and $\phi(\delta)$
are not connected by a $T$-curve on $\BS(s')$ in the following six cases:
$\gamma=(e,e,e)$, $\delta=(s_1,e,e)$; $\gamma=(e,e,e)$, $\delta=(e,s_4,e)$;
$\gamma=(e,e,s_3)$, $\delta=(s_1,e,s_3)$; $\gamma=(e,s_4,s_3)$, $\delta=(s_1,s_4,s_3)$;
$\gamma=(s_1,e,e)$, $\delta=(s_1,s_4,e)$; $\gamma=(e,s_4,e)$, $\delta=(s_1,s_4,e)$.
}
\end{example}

\subsection{Stabilization phenomenon} We are going to prove here that for any two $(p,w)$-pairs
$(\gamma,\delta)$ and $(\rho,\delta)$ with $\gamma\in\Gamma_s$ and $\rho\in\Gamma_t$, we get $s=t$.
To this end, we need the following result. In its proof, $\widehat x$ means the omission of
the factor $x$ in a product.

\begin{lemma}\label{lemma:1:new} Let $s=(s_1,\ldots,s_n)$ and $t=(t_1,\ldots,t_n)$ be two sequences
of simple reflections and $\gamma\in\Gamma_s,\rho\in\Gamma_t$ be two galleries.
Suppose that $\gamma^is_i(\gamma^i)^{-1}=\rho^i t_i(\rho^i)^{-1}$
for any $i=1,\ldots,n$. Then $s=t$.
\end{lemma}
\begin{proof} Note that we actually have
\begin{equation}\label{eq:-3:new}
\gamma^{i-1}s_i(\gamma^{i-1})^{-1}=\rho^{i-1}t_i(\rho^{i-1})^{-1}.
\end{equation}
We apply the induction on $n$, the case $n=1$ being trivial.

Let $n>1$ and we suppose that the lemma is true for smaller values of this parameter.
We write $s_i=s_{\alpha_i}$ and $t_i=s_{\tau_i}$ for the corresponding simple roots $\alpha_i$
and $\tau_i$.

We choose the maximal $k=1,\ldots n$ such that the following conditions hold:
{\renewcommand{\labelenumi}{{\rm \theenumi}}
\renewcommand{\theenumi}{{\rm(\roman{enumi})}}
\begin{enumerate}
\item\label{P:1} $s_i=t_i$ for $1\le i\le k$;
\item\label{P:2} $\gamma_i\ne\rho_i$ for $1\le i<k$;
\item\label{P:3} $s_i$ and $s_j$ commute for $1\le i<j\le k$.
\end{enumerate}}
As we have already seen that $s_1=t_1$, the value $k=1$
satisfies all three conditions. Thus our $k$ is well-defined.
Moreover, we assume that $k<n$, as otherwise we would be done.

First suppose that $\gamma_k=\rho_k$. For $i=k+1,\ldots,n$, we get
$$
\gamma_1\cdots\widehat\gamma_k\cdots\gamma_{i-1}s_i\gamma_{i-1}\cdots\widehat\gamma_k\cdots\gamma_1=
\rho_1\cdots\widehat\rho_k\cdots\rho_{i-1}t_i\rho_{i-1}\cdots\widehat\rho_k\cdots\rho_1.
$$
by~(\ref{eq:-3:new}) and~\ref{P:3}. These equalities together with~(\ref{eq:-3:new}) for $i=1,\ldots,k-1$ imply
$s_i=t_i$ for $i\in[1,n]\setminus\{k\}$ by the inductive hypothesis. However $s_k=t_k$ by~\ref{P:1}.

Now suppose that $\gamma_k\ne\rho_k$. By~\ref{P:1}, we have $\gamma_i\rho_i=s_i$ for any $i=1,\ldots,k$.
Thus~(\ref{eq:-3:new}) for $i=k+1$ can be written as follows:
\begin{equation}\label{eq:6}
s_1s_2\cdots s_ks_{k+1}s_k\cdots s_2s_1=t_{k+1}.
\end{equation}
By~\ref{P:3}, we can write this equality in terms of roots as follows:
\begin{equation}\label{eq:7}
\alpha_{k+1}-\<\alpha_{k+1},\alpha_1\>\alpha_1\pm\cdots\pm\<\alpha_{k+1},\alpha_k\>\alpha_k=\pm\tau_{k+1}.
\end{equation}

If $\alpha_{k+1}=\alpha_j$ for some $j=1,\ldots,k$, then $s_{k+1}=s_j$. This element commutes with all elements $s_1,\ldots,s_k$.
By~(\ref{eq:6}), we get $s_{k+1}=t_{k+1}$.
Hence conditions~\ref{P:1}--\ref{P:3} hold for $k+1$ instead of $k$, which contradicts its maximality.

We can suppose now that $\alpha_{k+1}\ne\alpha_j$ for $j=1,\ldots,k$.
It follows from~(\ref{eq:7}) that $\alpha_{k+1}=\tau_{k+1}$.
and $s_{k+1}=t_{k+1}$. As $k$ is maximal satisfying conditions~\ref{P:1}--\ref{P:3}, the reflection $s_{k+1}$
does not commute with some $s_j$, where $j\le k$. Hence $\<\alpha_{k+1},\alpha_j\>\ne0$.
For~(\ref{eq:7}) to hold, there must exist one more occurrence of $\alpha_j$ among $\alpha_1,\ldots,\alpha_k$.
That is $\alpha_j=\alpha_q$, where without loss of generality $1\le j<q\le k$.

Let $i=q+1,\ldots,n$. By~(\ref{eq:-3:new}) and~\ref{P:3}, we get
$$
s_1\cdots s_q\gamma_{q+1}\cdots \gamma_{i-1}
s_i
\gamma_{i-1}\cdots\gamma_{q+1}s_q\cdots s_1
=\rho_{q+1}\cdots\rho_{i-1}
t_i
\rho_{i-1}\cdots\rho_{q+1}.
$$
As $s_q=s_j$, we get again by~\ref{P:3} that
$$
s_1\cdots\widehat s_j\cdots s_{q-1}\gamma_{q+1}\cdots \gamma_{i-1}
s_i
\gamma_{i-1}\cdots\gamma_{q+1}s_{q-1}\cdots\widehat s_j\cdots s_1
=\rho_{q+1}\cdots\rho_{i-1}
t_i
\rho_{i-1}\cdots\rho_{q+1},
$$
Applying~\ref{P:3} once more, we get
\begin{equation}\label{eq:-4}
\begin{array}{l}
\gamma_1\cdots\widehat\gamma_j\cdots\widehat\gamma_q\cdots\gamma_{i-1}
s_i\gamma_{i-1}\cdots\widehat\gamma_q\cdots\widehat\gamma_j\cdots\gamma_1=\\[6pt]
\hspace{180pt}=
\rho_1\cdots\widehat\rho_j\cdots\widehat\rho_q\cdots\rho_{i-1}
t_i\rho_{i-1}\cdots\widehat\rho_q\cdots\widehat\rho_j\cdots\rho_1.
\end{array}
\end{equation}
Now let $i=j+1,\ldots,q-1$. By~\ref{P:3} and~\ref{P:1}, we get
$$
\gamma_1\cdots\widehat\gamma_j\cdots\gamma_{i-1}
s_i\gamma_{i-1}\cdots\widehat\gamma_j\cdots\gamma_1=
\rho_1\cdots\widehat\rho_j\cdots\rho_{i-1}
t_i\rho_{i-1}\cdots\widehat\rho_j\cdots\rho_1.
$$
These equalities together with equalities~(\ref{eq:-4}) for $i=q+1,\ldots,n$ and equalities~(\ref{eq:-3:new}) for $i=1,\ldots,j-1$ imply
$s_i=t_i$ for $i\in[1,n]\setminus\{j,q\}$ by the inductive hypothesis.
However $s_j=t_j$ and $s_q=t_q$ by~\ref{P:1}.
\end{proof}

\begin{corollary}\label{corollary:2}
Let $s$, $\tilde s$, $s'$ be sequences of simple reflections such that $|s|=|\tilde s|\le |s'|$.
Let $\delta\in\Gamma_{s'}$, $p:[1,|s|]\to[1,|s'|]$ be a monotone embedding and $w\in W$.
If there exist two morphisms $(p,w,\phi):s\to s'$ and $(p,w,\widetilde\phi):\tilde s\to s'$ such that
$\delta\in\im\phi\cap\im\widetilde\phi$,
then $s=\tilde s$ and there exists an isomorphism $(\id,e,\psi):s\to s$ such that the following diagram is commutative:
$$
\begin{tikzcd}
s\arrow{rr}{(\id,e,\psi)}\arrow{dr}[swap]{(p,w,\phi)}&&\arrow{dl}{(p,w,\tilde\phi)}s\\
  &s'
\end{tikzcd}
$$
\end{corollary}
\begin{proof} Let us write $\delta=\phi(\gamma)$ and $\delta=\widetilde\phi(\widetilde\gamma)$ for the corresponding
galleries $\gamma\in\Gamma_s$ and $\widetilde\gamma\in\Gamma_{\tilde s}$.
As $(\gamma,\delta)$ and $(\widetilde\gamma,\delta)$ are both $(p,w)$-pairs, the definition of Section~\ref{pw},
implies
$$
\gamma^is_i(\gamma^i)^{-1}=w^{-1}\delta^{p(i)}s'_{p(i)}(\delta^{p(i)})^{-1}w=\widetilde\gamma^i\tilde s_i(\widetilde\gamma^i)^{-1}.
$$
Hence $s=\tilde s$ by Lemma~\ref{lemma:1:new}. The above formula implies that $(\gamma,\widetilde\gamma)$ is an $(\id,e)$-pair.
By Lemma~\ref{lemma:0}, there exists a map $\psi:\Gamma_s\to\Gamma_s$ such that $\psi(\gamma)=\widetilde\gamma$
and $(\id,e,\psi):s\to s$ is a morphism of $\Fold$.
The equality $\widetilde\phi\psi(\gamma)=\phi(\gamma)$ can be extended to the equality of compositions $\widetilde\phi\psi=\phi$
by condition~\ref{F:3} of Section~\ref{morph}.

By Lemma~\ref{lemma:3}, the map $\psi$ is bijective. It is easy to prove that $(\id,e,\psi^{-1})$
is a morphism of $\Fold$, which is clearly inverse to $(\id,e,\psi)$
\end{proof}

\end{document}